\documentclass[12,letterpaper,twoside]{article}
\usepackage[utf8]{inputenc}
\usepackage[english]{babel}
\usepackage{amsmath, amsfonts, amssymb, mathrsfs, amsthm}
\usepackage{enumerate}
\usepackage{graphicx}
\usepackage{color}
\usepackage{mathtools}
\usepackage[paperwidth=170mm,paperheight=230mm, total={125mm, 170mm}, top=29mm,left=23mm,includeheadfoot]{geometry}
\usepackage{amsmath}
\usepackage{amsfonts}
\usepackage{amssymb}
\usepackage{fancyhdr}
\usepackage{verbatim}
\renewcommand{\thepage}{\arabic{page}}
\newtheorem{teo}{{\it \textbf{Theorem}}}
\newtheorem{lem}{{\it \textbf{Lemma}}}
\newtheorem{rem}{{\it \textbf{Remark}}}

\providecommand{\keywords}[1]
{
  \small	
  \textbf{\textit{Keywords---}} #1
}

\providecommand{\Classification}[1]
{
  \small	
  \textbf{\textit{Classification---}} #1
}
\title{{\Large Global dynamics of a two-strain flu model with a single vaccination and general incidence rate}}
\author{{\normalsize Arturo J. Nic May$^*$ and Eric J. Avila Vales$^*$}\\
        {\small$^*$ Facultad de Matemáticas, Universidad Autónoma de Yucatán, Anillo Periférico Norte,}\\{\small Tablaje 13615, Mérida, Yucatán C.P. 97119, México.} \\
        {\small  E-mail addresses: arturo\_javier\_1559@hotmail.mx, avila@correo.uady.mx} 
}
\date{} 
\begin{document}
\maketitle
\textbf{Resumen:} En este artículo tratamos la dinámica global de un modelo de gripe de dos cepas con vacuna solamente para la cepa 1 y una tasa de incidencia general.  La dinámica  global del modelo se determina a través de funciones adecuadas de Lyapunov. Ilustramos nuestros resultados por simulaciones numéricas.\\
\\  
\textbf{Abstract:} In this his paper, we studied the global dynamics of a two-strain flu model with a single-strain vaccine and general incidence rate. Four equilibrium points were obtained and the global dynamics of the model are completely determined via suitable lyapunov functions. We illustrate our results by some numerical si\-mu\-la\-ti\-ons.\\  
\\
\\
\Classification{34D23, 37B25, 92D30.}\\
\newpage
\section{Introduction}
Seasonal influenza is an acute respiratory infection caused by influenza viruses which circulate in all parts of the world. Worldwide, these annual epidemics are estimated to result in about 3 to 5 million cases of severe illness, and about 290 000 to 650 000 respiratory deaths [1]. This infection can have an endemic, epidemic or pandemic behavior.

There were, three major flu pandemics during the 20th century, the so called Spanish flu in
1918 had, been the most devastating pandemic. It has been estimated that the Spanish flu claimed
around 40–50 million deaths (as much as 3 \% of the total population), and it also infected 20–40\%
of the whole population. In 1957–1958, the Asian flu or bird flu pandemic caused more than two million deaths [2]. Unlike the Spanish flu, this time the infection-causing virus was detected earlier due to the advancement of scien\-ce and technology. A vaccine was made available but with limited supply.
After a decade (in 1968), a flu pandemic that originated again from Hong Kong hit mankind. That
flu pandemic also claimed one million lives. Beside these three major ones, there are some
other flu pandemics spreading among nations on smaller scales. For instance, the 2009 H1N1
swine flu is one of the more publicized pandemics that attracted the attention of all scientists and health professionals in the world and made them very much concerned. The pandemic, however,
did not result in great casualties like before. As of July 2010, only about 18,000 related deaths
had been reported [2]. There are many methods of preventing the spread of infectious
disease, one of them is vaccination. Vaccination is the ad\-mi\-nis\-tra\-tion of agent-specific, but relatively harmless, antigenic components that in vaccinated individuals can induce protective immunity against the corresponding infectious agent [3].

Influenza causes serious public-health problems around the world, therefore, we need to understand transmission mechanism and control strategies. Mathematical models also provided insight into severity of past influenza epidemics. Some models were used to investigate the three most de\-vas\-ta\-ting historical pandemics of influenza in the 20th century  [4–6]. There are a lot of  pathogens with several circulating strains. The presence of them is mostly due to incorrect treatment. 

Rahman and Zou [2] proposed a two-strain model with a single vaccination, namely. 
\begin{eqnarray}\label{tw1}
 \dot{S}&=&\Lambda- \left(\beta_1 I_1 + \beta_2 I_2 +\lambda \right)S \nonumber \\
 \dot{V_1}&=&rS-(\mu + k I_2)V_1\nonumber\\
 \dot{I_1}&=&\beta_1 I_1 S-\alpha_1 I_1 \nonumber \\
 \dot{I_2}&=&\beta_2 I_2 S+kI_2 V_1-\alpha_2 I_2 \nonumber \\
 \dot{R}&=&\gamma_1I_1+\gamma_2 I_2-\mu R.  
\end{eqnarray}
Where $\lambda=r+\mu$,  $\alpha_1=\gamma_1+v_1+\mu$,  $\alpha_2=\gamma_2+v_2+\mu$. The compartments are $S(t)$, $V_1(t)$, $I_1(t)$, $I_2(t)$ and  $R(t)$ which denote the population of susceptible, vaccine of strain 1, infective with respect to strain 1, infective with respect to strain 2 and removed individuals at time t, respectively. And  
\begin{itemize}
\item $\Lambda$ is the constant recruitment of individuals.
\item $\displaystyle \frac{1}{\mu}$ is the average time of life expectancy.
\item $r$ is the rate of vaccination with strain 1.
\item $k$ is the transmission coefficient of vaccinated individuals to strain 2.
\item $\beta_1$ is the transmission coefficient of susceptible individuals to strain 1.
\item $\beta_2$ is the transmission coefficient of susceptible individuals to strain 2.
\item $\displaystyle \frac{1}{\gamma_1}$ is the average infection period of strain 1.
\item $\displaystyle \frac{1}{\gamma_2}$ is the average infection period of strain 2.
\item $v_1$ is the infection-induced death rate of strain 1.
\item $v_2$ is the infection-induced death rate of strain 2.
\end{itemize}
The incidence rate of a disease measures how fast the disease is spreading and it plays
an important role in the research of epidemiology. Rahman and Zou [2] used the bilinear incidence rate $\beta SI$. However, there are more realistic incidence rates than the bilinear incidence rate, For instance, Capasso and his co-workers observed in the seventies [7] that
the incidence rate may increase more slowly as $I$ increases, so they proposed a saturated
incidence rate $ \frac{\beta I S}{1+\zeta I}$.

Baba and Hincal [8] studied an epidemic model consisting of three strains of influenza ($I_1$ , $I_2$, and $I_3$) where we have vaccine for strain1 ($V_1$) only, and  force of infection $\frac{\beta SI} {1 + \zeta S}$ for strain 2. Baba et al. [9] studied an studied an epidemic model consisting of two strains of influenza ($I_1$ and $I_2$) where  force of infection $\frac{\beta SI_2} {1 + \zeta I_2^2}$ for strain 2.

We propose to study model (1) modifying the force of infection in the compartments $I_1$ and $I_2$, by extending the incidence function to a more general form: 
\begin{equation}
F(S, I). \nonumber
\end{equation}
Which is based on the incidence rate studied in [10].

Thus, the resulting model is given by the following system:
\begin{eqnarray}\label{tw2}
 \dot{S}&=&\Lambda- F_1 (S,I_1) -F_2 (S,I_2) -\lambda S \nonumber \\
 \dot{V_1}&=&rS-(\mu + k I_2)V_1\nonumber\\
 \dot{I_1}&=&F_1 (S,I_1)-\alpha_1 I_1 \nonumber \\
 \dot{I_2}&=&F_2 (S,I_2)+kI_2 V_1-\alpha_2 I_2 \nonumber \\
 \dot{R}&=&\gamma_1I_1+\gamma_2 I_2-\mu R.  
\end{eqnarray}
Whose state space is  $\mathbb{R}^5_+= \{ (S, V_1, I_1,I_2, R ) : S \geq 0 , V_1 \geq 0, I_1\geq 0 ,I_2 \geq 0, R\geq 0 \}$ and subject to the initial conditions $S(0) = S_0 \geq 0$ , $V_1(0) = V_{10} \geq 0$ , $I_1(0) = I_{10} \geq 0$, $I_2(0) = I_{20} \geq 0$ and $R(0) = R_0 \geq 0$.

We make the following hypotheses on $F_i$, $i=1,2.$:
\begin{itemize}
\item[H1)]$F_i(S,I_i)=I_if_i(S,I_i)$ with $F_i$, $f_i\in \mathbf{C}^2(\mathbb{R}^2_+ \to \mathbb{R}_+)$ and $F(0,I_i)=F(S,0)$ for all $S,I_i\geq 0.$ 
\item[H2)]$\displaystyle \frac{\partial f_i}{\partial S}(S,I_i)>0$ and $\displaystyle \frac{\partial f_i}{\partial I_i}(S,I_i)\leq 0$ for all $S,I_i\geq 0.$
\item[H3)]$\displaystyle  \lim_{I_i\to 0^+} \frac{F_i(S,I_i)}{I_i}$ exist and is positive for all $S>0.$
\end{itemize}
The first of this hypotheses is a basic requirement for any biologically feasible
incidence rate, since the disease cannot spread when the number of
susceptible or infected individuals is zero.

As for (H2), the condition $\frac{\partial f_i}{\partial S} (S, I_i) > 0$ ensures the monotonicity of $f_i (S, I_i )$ on S, while $\frac{\partial f_i}{\partial I_i} (S, I ) \leq 0$suggests that $\frac{f_i(S, I_i )}{I_i}$ is non-increasing with respect to $I_i$ . In the case when $f_i$ monotonically increases with respect to both variables and is concave with respect to $I_i$ , the hypothesis (H2) naturally holds. Concave incidence functions have been used to represent the saturation effectin the transmission rate when the number of infectives is very high and exposure to thedisease is virtually certain.

(H3) is needed only to ensure that the basic reproduction number is well defined. Some examples of incidence functions studied in the literature that satisfy (H1)–(H3) are as follows:
\begin{itemize}
\item[(C1)]F(S,I)=$\beta SI$ [2]. 
\item[(C2)]F(S,I)=$\frac{\beta SI} {1 + \zeta S}$, where $\zeta>0$ describes the psychological effect of general public towards the infective [8].
\item[(C3)]F(S,I)=$\frac{\beta SI} {1 + \zeta I^2}$, where $\zeta>0$ measures the psychological or inhibitory effect of the population [9].
\end{itemize}
A more thorough list can be found in [10].

This paper is organized as follows. In section 2, we describe the disease dynamics described by the model. In section 3, we calculate the basic reproduction number. In section 4, we establish the existence of equilibrium points. In section 5,  we study the stability of the model. In section 6,  provides some numeric simulations to illustrate our main theoretical results. The paper ends with a some remarks.
\section{Disease dynamics described by the model}
From Model \eqref{tw2}, the total population $N = S+V_1+I_1+I_2+R$ satisfies:
\begin{eqnarray}
\dot{N}&=& \dot{S}+\dot{V_1}+\dot{I_1}+\dot{I_2}+\dot{R} \nonumber \\
&=& \Lambda-\mu S-\mu V_1-\mu I_1-\mu I_2 -\mu R -v_1 I_1-v_2 I_2 \nonumber\\
&\leq & \Lambda-\mu (S+ V_1+ I_1+ I_2 + R ) \nonumber\\
&=& \Lambda-\mu N. \nonumber
\end{eqnarray}
The comparison theorem then implies that $ \displaystyle \lim_{t \to \infty} \sup N(t)\leq \frac{\Lambda}{\mu}$.  Hence N(t) is bounded and so are all components $S(t)$, $V_1(t)$, $I_1(t)$, $I_2(t)$ and $R(t)$.

Since the equation for $\dot{R}$ is actually decoupled from the rest in equation \eqref{tw2}, we only need to consider dynamics of the following four-dimensional sub-system:
\begin{eqnarray} \label{tw3}
\dot{S}&=&\Lambda- F_1 (S,I_1) -F_1 (S,I_2) -\lambda S \nonumber \\
 \dot{V_1}&=&rS-(\mu + k I_2)V_1\nonumber\\
 \dot{I_1}&=&F_1 (S,I_1)-\alpha_1 I_1 \nonumber \\
 \dot{I_2}&=&F_2 (S,I_2)+kI_2 V_1-\alpha_2 I_2. 
\end{eqnarray}
Therefore, we have established the following.
\begin{lem}\label{lemaco}
For model \eqref{tw2}, the closed set 
\begin{equation}
\Omega =\left\lbrace (S,V_1,I_1,I_2)\in \mathbb{R}^4_+ N=S+V_1+I_1+I_2 \leq \frac{\Lambda}{\mu}\right\rbrace \nonumber
\end{equation}  
is  positively invariant.
\end{lem}
\section{Basic reproduction number}
The basic reproduction number of infection of model \eqref{tw3}, is a dimensionless quantity denoted by $\mathcal{R}_0$, and intuitively defined as the expected number of secondary infection cases caused by a single typical infective case during its entire period of infectivity in a wholly susceptible population. Then, referring to the method of [11].
\begin{eqnarray}
\mathcal{F}:= \left( \begin{matrix}
F_1(S,I_1)  \\
F_2(S,I_2)+kI_2 V_1
\end{matrix} \right). \nonumber
\end{eqnarray}
\begin{eqnarray}
\mathcal{V}:= \left( \begin{matrix}
\alpha_1 I_1 \\
\alpha_2 I_2
\end{matrix} \right). \nonumber
\end{eqnarray}
Then
\begin{eqnarray}
F'=\left. \left( \begin{matrix}
\frac{\partial F_1(S,I_1)}{\partial I_1} & 0 \\
 0 & \frac{\partial F_2(S,I_2)}{\partial I_2}+k V_1
\end{matrix} \right)\right|_{E_0}
= \left( \begin{matrix}
\frac{\partial F_1(S_0,0)}{\partial I_1} & 0 \\
 0 & \frac{\partial F_2(S_0,0)}{\partial I_2}+\frac{k r \Lambda}{\mu \lambda}
\end{matrix} \right) .\nonumber
\end{eqnarray}
\begin{eqnarray}
V'=\left. \left( \begin{matrix}
\alpha_1  & 0 \\
 0&\alpha_2 
\end{matrix} \right)\right|_{E_0}
= \left( \begin{matrix}
\alpha_1  & 0 \\
 0&\alpha_2 
\end{matrix} \right). \nonumber
\end{eqnarray}

Where $\displaystyle E_0=(S^0,V_{1}^0,0,0)=\left(\frac{\Lambda}{\lambda}, \frac{r \Lambda}{\mu \lambda}, 0,0 \right)$. The matrix F is non-negative and is responsible for new infections, while the V is in\-ver\-ti\-ble and is referred to as the transmission matrix for the model \eqref{tw3}. It follows that,
\begin{eqnarray}
F'V'^{-1}== \left( \begin{matrix}
\frac{\sigma_1}{\alpha_1} & 0 \\
 0 & \frac{\sigma_2}{\alpha_2}+\frac{k r \Lambda}{\alpha_2 \mu \lambda}
\end{matrix} \right). \nonumber
\end{eqnarray}
Where $\sigma_i=\displaystyle \frac{\partial F_i(S_0,0)}{\partial I_i}$, for $i=1,2$. Thus, the basic reproduction number can be calculate as\\
$$\mathcal{R}_0=\rho (F'V'^{-1})=\max \displaystyle \left\lbrace \frac{\sigma_1}{\alpha_1}, \ \frac{\sigma_2}{\alpha_2}+\frac{k r \Lambda}{\alpha_2 \mu \lambda} \right\rbrace.$$
Where $\rho( A )$ denotes the spectral radius of a matrix A. Let
\begin{center}
$\mathcal{R}_1=\frac{\sigma_1}{\alpha_1}$ and $\mathcal{R}_2=\frac{\sigma_2 }{\alpha_2}+\frac{k r \Lambda}{\alpha_2 \mu \lambda}.$ 
\end{center}
Then 
$$\mathcal{R}_0=\max \{\mathcal{R}_1, \mathcal{R}_2 \}.$$
Therefore $\mathcal{R}_1, \mathcal{R}_2\leq \mathcal{R}_0.$
\section{Existence of equilibrium solutions}
The four possible equilibrium points for the system \eqref{tw3} are: Disease-free equilibrium, single-strain $(I_1)$-infection, single-strain $(I_2)$-infection and endemic equilibrium. The system \eqref{tw3} has disease-free equilibrium $E_0 =\left(\frac{\Lambda}{\lambda}, \frac{r \Lambda}{\mu \lambda}, 0,0 \right)$ for all parameter values. We will now prove the existence of  the others equilibrium points. First we will show some lemmas.
\begin{lem}\label{lemFides}
By i=1,2. $$\frac{\partial F_i(S,I_i)}{\partial I_i}=I\frac{\partial f_i(S,I_i)}{\partial I_i}+\frac{F_i(S,I_i)}{I_i}.$$\\
Also:
$$\frac{\partial F_i(S,I_i)}{\partial I_i}\leq \frac{F_i(S,I_i)}{I_i}.$$
\end{lem} 
\begin{proof}
By H1)
$$F_i(S,I_i)=I_i f_i(S,I_i)$$
Then
\begin{eqnarray}
\frac{\partial F_i(S,I_i)}{\partial I_i}&=&I_i\frac{\partial f_i(S,I_i)}{\partial I_i}+f_i(S,I_i)\nonumber \end{eqnarray}
By H2) $\displaystyle \frac{\partial f_i(S,I_i)}{\partial I_i}\leq 0$, then:
\begin{eqnarray}
\frac{\partial F_i(S,I_i)}{\partial I_i} \leq f_1(S,I_i)=\frac{F_i(S,I_i)}{I_i}.\nonumber
\end{eqnarray}
\end{proof}
\begin{lem}
By model \eqref{tw3}, the closed set  $\Omega_1=\left\lbrace (S,V_1,I_1,I_2)\in \Omega | S \leq S^0 \ \text{and} \  V_1\leq V_1^0 \right\rbrace$ is a positively invariant set.
\end{lem}
\begin{proof}
As $\Omega$ is a positively invariant set for model (4), it will be enough to show that if $S=S^0$, then $\dot{S}\leq 0$ and if $S\leq S^0$ and  $V_1=V_1^0$, then $\dot{V_1}\leq 0.$\\ 
If $S=S^0$, then
\begin{eqnarray}
\dot{S}&=&\Lambda- F_1 (S^0,I_1) -F_1 (S^0,I_2) -\lambda S^0\nonumber\\
&=&\lambda S^0- F_1 (S^0,I_1) -F_1 (S^0,I_2) -\lambda S^0\nonumber\\
&=&- F_1 (S^0,I_1) -F_1 (S^0,I_2)\leq 0 \nonumber
\end{eqnarray}
If $S\leq S^0$ and $V_1=V_1^0$, Then
\begin{eqnarray}
 \dot{V_1}&\leq& rS^0-(\mu + k I_2)V_1^0\nonumber\\
 &=& rS^0-\mu V_1^0 - k I_2 V_1^0=- k I_2 V_1^0\leq 0 \nonumber
\end{eqnarray}
\end{proof}
\begin{lem}\label{lemfs}
By i=1,2. $$\frac{\partial F_i(S,I_i)}{\partial S}\geq0$$
\end{lem}
\begin{proof} 
By H1)
$$F_i(S,I_i)=I_i f_i(S,I_i)$$
Then
\begin{eqnarray}
\frac{\partial F_i(S,I_i)}{\partial S}&=&I_i\frac{\partial f_i(S,I_i)}{\partial S}\geq 0 \ \text{By H2)}.\nonumber 
\end{eqnarray}
\end{proof}
\begin{rem}\label{remtw1}
By H2) given $a$ and $b$, for all $S$ and $I_i$ if  $S \leq a$,  $I_i \geq b$, then 
$f_i(S,I_i)\leq f_i(a,b)$, $i=1,2$.
\end{rem}
\begin{teo}\label{teoext}
\begin{itemize}
\item[1)]The model \eqref{tw3} admits a unique single-strain $(I_1)$-infection equilibrium $E_1=(\bar{S},\bar{V_1},\bar{I_1},0)$ if and only if $\mathcal{R}_1>1$. 
\item[2)]The model \eqref{tw3} admits a single-strain $(I_2)$-infection equilibrium $E_2=(\tilde{S},\tilde{V_1},0,\tilde{I_2})$ if and only if $\mathcal{R}_2>1$. Also, if $-\alpha_2 r \mu-\alpha_2 \mu^2+k\Lambda r<0$ then $E_2$ is unique. While if  $-\alpha_2 r \mu-\alpha_2 \mu^2+k\Lambda r>0$ then the model (3) has at most one single-strain $(I_2)$-infection in the interval $\left[\frac{-r\alpha_2-\alpha_2\mu +\sqrt{r\alpha_2(r\alpha_2+\alpha_2 \mu+k\Lambda)}}{\alpha_2k} ,\frac{\Lambda}{\alpha_2}\right]$. 
\end{itemize}
\end{teo}
\begin{proof}
\begin{itemize}
\item[1)] 
If $I_2=0$ and $\mathcal{R}_1>1$, we consider the system 
\begin{eqnarray}
&&\Lambda-F_1(\bar{S},\bar{I_1})-\lambda \bar{S}=0 \label{tw4}\\
&&r \bar{S}-\mu \bar{V_1} =0  \label{tw5}\\
&&F_1(\bar{S},\bar{I_1})-\alpha_1 \bar{I_1}=0  \label{tw6}. 
\end{eqnarray} 
By \eqref{tw5} and \eqref{tw6} 
\begin{equation}
\bar{V_1}=\frac{r \bar{S}}{\mu},\ F_1(\bar{S},\bar{I_1})=\alpha_1 \bar{I_1}. \nonumber 
\end{equation}
Substituting in \eqref{tw4}.
\begin{eqnarray}
\Lambda-\alpha_1 \bar{I_1}-\lambda \bar{S}=0\nonumber \\
\bar{S}=\frac{\Lambda-\alpha_1 \bar{I_1}}{\lambda}. \nonumber
\end{eqnarray}
Note that $S\geq 0$ if and only if $\bar{I_1}\leq \frac{\Lambda}{\alpha_1}$. $\bar{I_1}$ being determined by the positive roots of the equation.
\begin{equation}
G(\bar{I_1})\equiv F_1(\frac{\Lambda-\alpha_1 \bar{I_1}}{\lambda},\bar{I_1})-\alpha_1 \bar{I_1}. \label{tw7} 
\end{equation}
See that
\begin{equation}
G'(\bar{I_1})=\frac{-\alpha_1}{\lambda} \frac{\partial F_1(\frac{\Lambda-\alpha_1 \bar{I_1}}{\lambda},\bar{I_1})}{\partial S}+\frac{\partial F_1(\frac{\Lambda-\alpha_1 \bar{I_1}}{\lambda},\bar{I_1})}{\partial I_1}-\alpha_1.  \nonumber
\end{equation}
Then 
\begin{equation}
G(0)=F_1(\frac{\Lambda}{\lambda},0)=0\ \text{by H1}. \nonumber
\end{equation}
\begin{eqnarray}
G'(0)&=&\frac{-\alpha_1}{\lambda} \frac{\partial F_1(\frac{\Lambda}{\lambda},0)}{\partial S}+\frac{\partial F_1(\frac{\Lambda}{\lambda},0)}{\partial I_1}-\alpha_1  \nonumber\\
&=&\frac{\partial F_1(S_0,0)}{\partial I_1}-\alpha_1 \nonumber \ \text{by H1}\\
&=&\alpha_1 \left(\frac{ \sigma_1}{\alpha_1}-1\right)=\alpha_1\left(\mathcal{R}_1-1\right)>0. \nonumber
\end{eqnarray}
Therefore $G(\bar{I_1})>0$ by $I_1$ sufficiently small. Also
\begin{equation}
G(\frac{\Lambda}{\alpha_1})=F_1(0,\bar{I_1})-\Lambda=-\Lambda<0. \nonumber
\end{equation}
then equation \eqref{tw7} has a positive root.

Also if $ E_1 $ exists then
\begin{eqnarray}
f_1(\bar{S},\bar{I_1})-\alpha_1 =0.\nonumber
\end{eqnarray}
Note that $\bar{S}< S^0$. Then by lemma \ref{lemFides} and remark \ref{remtw1}
\begin{eqnarray}
0 &< & f_1(S^0,0)-\alpha_1\nonumber\\
&=&\frac{\partial F_1(S^0,0)}{\partial I_1}-\alpha_1\nonumber\\  
&=&\alpha_1 \left(\mathcal{R}_1-1\right).\nonumber  
\end{eqnarray}
Then $\mathcal{R}_1>1.$

Next, we shall show that $\bar{I_1}$ is unique. From \eqref{tw6}, it follows that
\begin{eqnarray}
\alpha_1=f_1(\bar{S},\bar{I_1}) \nonumber
\end{eqnarray}
Using (H2) and lemma \ref{lemFides}, we have that $\frac{-\alpha_1}{\lambda} \frac{\partial F_1(\bar{S},\bar{I_1})}{\partial S}\leq 0$ and $\bar{I_1}\frac{\partial f_1(\bar{S},\bar{I_1})}{\partial I_1}<0$. Furthermore, it can be found that
\begin{eqnarray}
G'(\bar{I_1})&=&\frac{-\alpha_1}{\lambda} \frac{\partial F_1(\frac{\Lambda-\alpha_1 \bar{I_1}}{\lambda},\bar{I_1})}{\partial S}+\frac{\partial F_1(\frac{\Lambda-\alpha_1 \bar{I_1}}{\lambda},\bar{I_1})}{\partial I_1}-\alpha_1.  \nonumber\\
&=&\frac{-\alpha_1}{\lambda} \frac{\partial F_1(\frac{\Lambda-\alpha_1 \bar{I_1}}{\lambda},\bar{I_1})}{\partial S}+\bar{I_1}\frac{\partial \bar{I_1}(\bar{S},\bar{I_1})}{\partial \bar{I_1}}+f_1(\bar{S},\bar{I_1})-f_1(\bar{I_1},\bar{I_1})\nonumber\\
&=&\frac{-\alpha_1}{\lambda} \frac{\partial F_1(\frac{\Lambda-\alpha_1 \bar{I_1}}{\lambda},\bar{I_1})}{\partial S}+\bar{I_1}\frac{\partial \bar{I_1}(\bar{S},\bar{I_1})}{\partial \bar{I_1}}<0.\nonumber
\end{eqnarray}
Which implies that $G(\bar{I_1})$ strictly decreases at any of the zero points of \eqref{tw7}. Let us
suppose that \eqref{tw7} has more than one positive root. Without loss of generality, we
choose the one, denoted by $\bar{I_1}^*$, that is the nearest to $\bar{I_1}$. Because of the continuity of $G(\bar{I_1} )$, we must have $G'(\bar{I_1}^*)\geq 0$, which results in a contraction with the strictly decreasing property of $G(\bar{I_1})$ at all the zero points. 
\item[2)]If $I_1=0$ and $\mathcal{R}_2>1$, we consider the system
\begin{eqnarray}
\Lambda-F_2(\tilde{S},\tilde{I_2})-\lambda \tilde{S}=0 \label{tw8}\\
r \tilde{S}-(\mu+k\tilde{I_2}) \tilde{V_1} =0\label{tw9}\\
F_2(\tilde{S},\tilde{I_2})+k\tilde{I_2}\tilde{V_1}-\alpha_2 \tilde{I_2}=0.\label{tw10} 
\end{eqnarray} 
By \eqref{tw9} and \eqref{tw10}
\begin{equation}
\tilde{V_1}=\frac{r \tilde{S}}{\mu+k\tilde{I_2}},\ F_2(\tilde{S},\tilde{I_2})=-k\tilde{I_2}\tilde{V_1}+\alpha_2 \tilde{I_2}. \nonumber 
\end{equation}
Substituting in \eqref{tw8}.
\begin{eqnarray}
\Lambda-\alpha_2 \tilde{I_2}+k\tilde{I_2}\tilde{V_1}-\lambda \tilde{S}=0\nonumber \\
\left(\lambda-\frac{ kr \tilde{I_2}}{\mu +k \tilde{I_2}}  \right)\tilde{S}=\Lambda-\alpha_2 \tilde{I_2}.\nonumber 
\end{eqnarray}
\begin{eqnarray}
\left(\frac{\lambda(\mu +k \tilde{I_2}) -kr \tilde{I_2}}{\mu +k \tilde{I_2}}  \right)\tilde{S}=\Lambda-\alpha_2 \tilde{I_2}\nonumber \\
\left(\frac{\lambda \mu +(\mu+r)k \tilde{I_2} -kr \tilde{I_2}}{\mu +k \tilde{I_2}}  \right)\tilde{S}=\Lambda-\alpha_2 \tilde{I_2}\nonumber \\
\tilde{S}=\left(\Lambda-\alpha_2 \tilde{I_2}\right)
\left(\frac{\mu +k \tilde{I_2}}{\lambda \mu +\mu k \tilde{I_2}}  \right).
\nonumber
\end{eqnarray}
Note that $\tilde{S}\geq 0$ if and only if $\tilde{I_2}\leq \frac{\Lambda}{\alpha_2}$. $\tilde{I_2}$ being determined by the positive roots of the equation.
\begin{eqnarray}
H(\tilde{I_2})&\equiv & F_2\left(\frac{(\Lambda-\alpha_2 \tilde{I_2})(\mu +k \tilde{I_2})}{\lambda \mu +k \mu \tilde{I_2}},\bar{I_2}\right)+k\tilde{I_2}\tilde{V_1}-\alpha_2 \tilde{I_2}\nonumber \\
&=&F_2\left(\frac{\Lambda \mu +(\Lambda k -\alpha_2 \mu ) \tilde{I_2} -k \alpha_2 \tilde{I_2}^2}{\lambda \mu +k \mu \tilde{I_2}},\tilde{I_2}\right)\nonumber\\
&&+\left(\frac{ \Lambda rk\tilde{I_2} -\alpha_2 r k \tilde{I_2}^2}{\lambda \mu +k\mu \tilde{I_2}}\right)-\alpha_2 \bar{I_2}. \label{tw11} 
\end{eqnarray}
See that 
\begin{eqnarray}
H'(\tilde{I_2})&=&\frac{(k \mu)(-k \alpha_2 \tilde{I_2}^2-\Lambda\mu)+ \lambda\mu(\Lambda k -\alpha_2 \mu-2k \alpha_2 \tilde{I_2})}{\left(\lambda \mu +\mu k \tilde{I_2} \right)^2} \nonumber\\
&&\times   \frac{\partial F_2\left(\frac{\Lambda \mu +(\Lambda k -\alpha_2 \mu ) \tilde{I_2} -k \alpha_2 \tilde{I_2}^2}{\lambda \mu +\mu k \tilde{I_2}},\tilde{I_2}\right)}{\partial S}\nonumber\\
&&+\frac{\partial F_2\left(\frac{\Lambda \mu +(\Lambda k -\alpha_2 \mu ) \tilde{I_2} -k \alpha_2 \tilde{I_2}^2}{\lambda \mu +\mu k \tilde{I_2}},\tilde{I_2}\right)}{\partial I_1}\nonumber\\
&&+\left(\frac{ \lambda\mu(\Lambda rk-\alpha_2 r k \tilde{I_2}) -(k \mu)\alpha_2 r k \tilde{I_2}^2}{(\lambda \mu +\mu k \tilde{I_2})^2}\right)-\alpha_2.  \nonumber
\end{eqnarray}
Then 
\begin{equation}
H(0)=F_2\left(\frac{\Lambda}{\lambda},0\right)=0\ \text{by H1}. \nonumber
\end{equation}
\begin{eqnarray}
H'(0)&=&\frac{\partial F_2(S_0,0)}{\partial I_1}+\frac{\Lambda r k}{\lambda \mu}-\alpha_2 \nonumber \ \text{by H1}\\
&=&\alpha_2 \left(\frac{ \beta_2}{\alpha_2}+\frac{\Lambda r k}{\alpha_2 \lambda \mu}-1\right)=\alpha_2 \left( \mathcal{R}_2-1\right)>0. \nonumber
\end{eqnarray}
Therefore $H(\bar{I_2})>0$ by $I_2$ sufficiently small. Also
\begin{equation}
H\left(\frac{\Lambda}{\alpha_2}\right)=F_2\left(0,\frac{\Lambda}{\alpha_2} \right)-\Lambda=-\Lambda<0. \nonumber
\end{equation}
then equation \eqref{tw11} has a positive root.

Also if $ E_2 $ exists then
\begin{eqnarray}
f_2(\tilde{S},\tilde{I_2})+k\tilde{V_1}-\alpha_2 =0.\nonumber
\end{eqnarray}
Note that $\tilde{S}< S^0$ and $\tilde{V_1}< V_1^0$. Then by lemma \ref{lemFides} and remark \ref{remtw1}
\begin{eqnarray}
0 &< & f_2(S^0,0)+kV_1^0-\alpha_2\nonumber\\
&=&\frac{\partial F_2(S^0,0)}{\partial I_2}+kV_1^0-\alpha_2\nonumber\\  
&=&\alpha_2\left(\mathcal{R}_2-1\right).\nonumber  
\end{eqnarray}
Then $\mathcal{R}_2>1.$

Next, we shall show that $\tilde{I_2}$ is unique if  $-\alpha_2 r \mu-\alpha_2 \mu^2+k\Lambda r<0$ and if  $-\alpha_2 r \mu-\alpha_2 \mu^2+k\Lambda r>0$ then the model (3) has at most one single-strain $(I_2)$-infection in the interval $\left[\frac{-r\alpha_2-\alpha_2\mu +\sqrt{r\alpha_2(r\alpha_2+\alpha_2 \mu+k\Lambda)}}{\alpha_2k} ,\frac{\Lambda}{\alpha_2}\right]$.

From \eqref{tw10}, it follows that
\begin{eqnarray}
\alpha_2-k\tilde{V_1}=f_2(\tilde{S},\tilde{I_2}). \nonumber
\end{eqnarray}
Furthermore, it can be found that
\begin{eqnarray}
H'(\tilde{I_2})&=&\frac{-\alpha_2 r \mu-\alpha_2 \mu^2-2\alpha_2 \mu k\tilde{I_2}-2 \alpha_2 k r \tilde{I_2} -\alpha_2 k^2 \tilde{I_2}^2+k \Lambda r}{\mu(\lambda + k \tilde{I_2})^2}\nonumber\\
&&\times \frac{\partial F_2(\tilde{S},\tilde{I_2})}{\partial S}+\frac{\partial F_2(\tilde{S},\tilde{I_2})}{\partial I_2}+k\tilde{V_1}-\frac{r(\alpha_2 \lambda+k \Lambda)\tilde{I_2}}{\mu(\lambda+k \tilde{I_2})^2}-\alpha_2  \nonumber\\
&=&\frac{-\alpha_2 r \mu-\alpha_2 \mu^2-2\alpha_2 \mu k\tilde{I_2}-2 \alpha_2 k r \tilde{I_2} -\alpha_2 k^2 \tilde{I_2}^2+k \Lambda r}{\mu(\lambda + k \tilde{I_2})^2}\nonumber\\
&&\times \frac{\partial F_2(\tilde{S},\tilde{I_2})}{\partial S}+\tilde{I_2}\frac{\partial f_2(\tilde{S},\tilde{I_2})}{\partial I_2}-\frac{r(\alpha_2 \lambda+k \Lambda)\tilde{I_2}}{\mu(\lambda+k \tilde{I_2})^2}  \nonumber.
\end{eqnarray}
If  $-\alpha_2 r \mu-\alpha_2 \mu^2+k\Lambda r<0$, then $H'(\tilde{I_2})<0$  which implies that $H(\tilde{I_2})$ strictly decreases at any of the zero points of \eqref{tw11}. Let us
suppose that \eqref{tw11} has more than one positive root. Without loss of generality, we
choose the one, denoted by $\tilde{I_2}^*$, that is the nearest to $\tilde{I_2}$. Because of the continuity of $H(\tilde{I_2})$, we must have $H'(\tilde{I_2}^*)\geq 0$, which results in a contraction with the strictly decreasing property of $H(\tilde{I_2})$ at all the zero points.

If  $-\alpha_2 r \mu-\alpha_2 \mu^2+k\Lambda r>0$, Let us suppose that \eqref{tw11} has more than one positive root in $\left[\frac{-r\alpha_2-\alpha_2\mu +\sqrt{r\alpha_2(r\alpha_2+\alpha_2 \mu+k\Lambda)}}{\alpha_2k} ,\frac{\Lambda}{\alpha_2}\right]$.  Without loss of generality, we choose the one, denoted by $\tilde{I_2}^*$, that is the nearest to $\tilde{I_2}$. Note that $H'(\tilde{I_2}^*)<0$ and $H'(\tilde{I_2}^*)<0$. Because of the continuity of $H(\tilde{I_2})$, we must have $H'(\tilde{I_2}^*)\geq 0$, which results in a contraction.    
\end{itemize}
\end{proof}
\begin{teo}\label{teoext2}
If $\displaystyle \frac{\partial F_2(S,I_2)}{\partial S}\leq I_2$, $\forall S,I_2$. Then the model \eqref{tw3} admits a unique single-strain $(I_2)$-infection equilibrium $E_2=(\tilde{S},\tilde{V_1},0,\tilde{I_2})$ if and only if $\mathcal{R}_2>1$.
\end{teo}
\begin{proof}
Similar argument to the proof of Theorem \ref{teoext} proof that the model \eqref{tw3} admits a single-strain $(I_2)$-infection equilibrium $E_2=(\tilde{S},\tilde{V_1},0,\tilde{I_2})$ if and only $\mathcal{R}_2>1$. Also If  $\displaystyle \frac{\partial F_2(S,I_2)}{\partial S}\leq I_2$, then

\begin{eqnarray}
H'(\tilde{I_2})&\leq& \frac{-\alpha_2 r \mu-\alpha_2 \mu^2-2\alpha_2 \mu k\tilde{I_2}-2 \alpha_2 k r \tilde{I_2} -\alpha_2 k^2 \tilde{I_2}^2}{\mu(\lambda + k I_2)^2}\nonumber\\
&&\times \frac{\partial F_2(\tilde{S},\tilde{I_2})}{\partial S}+\tilde{I_2}\frac{\partial f_2(\tilde{S},\tilde{I_2})}{\partial I_2}-\frac{r(\alpha_2 \lambda)\tilde{I_2}}{\mu(\lambda+k \tilde{I_2})^2}  \nonumber.
\end{eqnarray}
Then $H'(\tilde{I_2})<0$  which implies that $H(\tilde{I_2})$ strictly decreases at any of the zero points of \eqref{tw11}. Let us
suppose that \eqref{tw11} has more than one positive root. Without loss of generality, we
choose the one, denoted by $\tilde{I_2}^*$, that is the nearest to $\tilde{I_2}$. Because of the continuity of $H(\tilde{I_2})$, we must have $H'(\tilde{I_2}^*)\geq 0$, which results in a contraction with the strictly decreasing property of $H(\tilde{I_2})$ at all the zero points.
\end{proof}
\begin{rem}\label{remtw2}
Some examples of incidence functions that satisfy $\displaystyle \frac{\partial F_2(S,I_2)}{\partial S}\leq I_2$ are (C1), (C2) and (C3) when $\beta\leq 1$.
\end{rem}
The model \eqref{tw3} can have endemic infection equilibrium $E_3 =(S^*,V_1^*,I_1^*,I_2^*)$. To find $E_3$, we consider the system
\begin{eqnarray}
\Lambda-F_1(S^*,I_1^*)-F_2(S^*,I_2^*)-\lambda S^*=0 \label{tw12}\\
r S^*-(\mu+k I_2^*) V_1^* =0 \label{tw13}\\
F_1(S^*,I_1^*)-\alpha_1 I_1^*=0 \label{tw14}\\
F_2(S^*,I_2^*)+k I_2^* V_1^*-\alpha I_2^*=0. \label{tw15}
\end{eqnarray} 
By \eqref{tw13}, \eqref{tw14} and \eqref{tw15}
\begin{equation}
V_1^*=\frac{r S^*}{\mu+k I_2^*},\ F_1(S^*,I_1^*)=\alpha_1 I_1^*,\ F_2( S^*, I_2^* )=-k I_2^* V_1^*+\alpha_2 I_2^*. \nonumber 
\end{equation}
Substituting in \eqref{tw12}.
\begin{eqnarray}
&&\Lambda-\alpha_1 I_1^*-\alpha_2 I_2^* +k I_2^* V_1^*-\lambda  S^*=0\nonumber \\
&&\left(\lambda-\frac{ kr  I_2^*}{\mu +k  I_2^*}  \right)S^*=\Lambda-\alpha_1 I_1^*-\alpha_2 I_2^*\nonumber\\
&&\left(\frac{\lambda \mu +(\mu+r)k  I_2^* -kr  I_2^*}{\mu +k  I_2^*}  \right) S^*=\Lambda-\alpha_2  I_2^*\nonumber \\
&&S^*=\left(\Lambda-\alpha_1 I_1^*-\alpha_2  I_2^* \right)
\left(\frac{\mu +k I_2^*}{\lambda \mu +\mu k I_2^*}  \right).
\nonumber
\end{eqnarray}
Note that $S^* \geq 0$ if and only if $I_1^* \leq \frac{\Lambda-\alpha_2 I_2^*}{\alpha_1}$ and $I_2^* \leq \frac{\Lambda-\alpha_1 I_1^*}{\alpha_2}$. $\bar{I_2}$ being determined by the positive roots of the equation.
\begin{eqnarray}
G_2(I_2^*)&\equiv & f_2\left(\frac{(\Lambda-\alpha_1 I_1^*-\alpha_2 I_2^*)(\mu +k I_2^*)}{\lambda \mu +k \mu I_2^*}, I_2^* \right)+k V_1^*-\alpha_2. \nonumber 
\end{eqnarray}

$I_1^*$ being determined by the positive roots of the equation.
\begin{equation}
G_1({I_1^*})\equiv f_1\left(\frac{(\Lambda-\alpha_1 I_1^*-\alpha_2 I_2^*)(\mu +k I_2^*)}{\lambda \mu +k \mu I_2^*},{I_1^*}\right)-\alpha_1.  \nonumber
\end{equation}
\section{Stability of equilibrium}
In this section we will study the local and global stability of the equilibrium points. 
\begin{teo}\label{teole0}
The disease-free equilibrium $E_0=\displaystyle \left(\frac{\Lambda}{\lambda}, \frac{r \Lambda}{\mu \lambda}, 0,0 \right)$ is unstable if $\mathcal{R}_0 > 1$ while it is locally asymptotically stable if $\mathcal{R}_0 < 1$.
\end{teo}
\begin{proof}
The Jacobian matrix of the model, we get as follows:
\begin{eqnarray}\label{tw16}
J:= \left( \begin{matrix}
-\frac{\partial F_1}{\partial S} - \frac{\partial F_2}{\partial S} -\lambda  & 0 &-\frac{\partial F_1}{\partial I_1} &-\frac{\partial F_2}{\partial I_2} \\
r & -\mu-k I_2&0& -k V_1 \\
\frac{\partial F_1}{\partial S}& 0 & \frac{\partial F_1}{\partial I_1}-\alpha_1 & 0\\
\frac{\partial F_2}{\partial S} & k I_2 & 0&\frac{\partial F_2}{\partial I_2}+k V_1-\alpha_2  
\end{matrix} \right).
\end{eqnarray}
Then Eq. \eqref{tw16} at the disease-free equilibrium $E_0$  is
\begin{eqnarray}
J_{E_0}&=& \left( \begin{matrix}
-\frac{\partial F_1(S^0,0)}{\partial S} - \frac{\partial F_2(S^0,0)}{\partial S} -\lambda  & 0 &-\frac{\partial F_1(S^0,0)}{\partial I_1} &-\frac{\partial F_2(S^0,0)}{\partial I_2} \\
r & -\mu &0& -k V_1^0 \\
\frac{\partial F_1(S^0,0)}{\partial S}& 0 & \frac{\partial F_1(S^0,0)}{\partial I_1}-\alpha_1 & 0\\
\frac{\partial F_2(S^0,0)}{\partial S} & 0 & 0&\frac{\partial F_2(S^0,0)}{\partial I_2}+k V_1^0-\alpha_2  
\end{matrix} \right) \nonumber \\
&=& \left( \begin{matrix}
-\lambda  & 0 &-\frac{\partial F_2(S^0,0)}{\partial I_1}  &-\frac{\partial F_2(S^0,0)}{\partial I_2} \\
r & -\mu&0& -k V_1^0 \\
0& 0 & \frac{\partial F_1(S^0,0)}{\partial I_1} -\alpha_1 & 0\\
0& 0 & 0 & \frac{\partial F_2(S^0,0)}{\partial I_2}+\frac{k r \Lambda}{\mu \lambda}-\alpha_2 
\end{matrix} \right)\nonumber\\
&=& \left( \begin{matrix}
-\lambda  & 0 &-\sigma  &-\sigma_2 \\
r & -\mu&0& -k V_1^0 \\
0& 0 & \alpha_1\left(\frac{\sigma_1}{ \alpha_1} -1\right) & 0\\
0& 0 & 0 & \alpha_2 \left(\frac{\sigma_2 }{ \alpha_2}+\frac{k r \Lambda}{\mu \lambda \alpha_2}-1\right) 
\end{matrix} \right) \nonumber\\
&=& \left( \begin{matrix}
-\lambda  & 0 &-\sigma_1  &-\sigma_2 \\
r & -\mu&0& -k V_1 \\
0& 0 & \alpha_1\left(\mathcal{R}_1 -1\right) & 0\\
0& 0 & 0 & \alpha_2 \left(\mathcal{R}_2-1\right) 
\end{matrix} \right)\label{tw17}.
\end{eqnarray}
Thus the eigenvalues of the above Eq. \eqref{tw17} are
\begin{equation}
\lambda_1=-\lambda, \ \lambda_2=-\mu,\  \lambda_3=\alpha_1(\mathcal{R}_1 -1), \ \lambda_4=\alpha_2(\mathcal{R}_2 -1). \label{tw18}  
\end{equation}
From \eqref{tw18}, if $\mathcal{R}_0 < 1$, then $\lambda_3,\lambda_4 < 0$ and we obtain that the disease-free equilibrium $E^0$ of Model \eqref{tw3} is locally asymptotically stable. If $\mathcal{R}_0 > 1$, then the disease-free equilibrium loses its stability.
\end{proof}
\begin{teo}\label{teole1}
Let $\bar{\mathcal{R}_2}=\frac{1}{\alpha_2}\frac{\partial F_2(\bar{S},0)}{\partial I_2}+\frac{k \bar{V_1}}{\alpha_2}$. The equilibrium $E_1$ is unstable if $\bar{\mathcal{R}_2} > 1$ while it is locally asymptotically stable if $ 1<\bar{\mathcal{R}_2}$.
\end{teo}
\begin{proof}
Then Eq. \eqref{tw16} at the equilibrium $E_1$  is
\begin{eqnarray}
J_{E_1}&=& \left( \begin{matrix}
A_{11}   & 0 &A_{13} &A_{14} \\
r & -\mu &0& A_{24} \\
A_{31}& 0 & A_{33} & 0\\
0 & 0 & 0&A_{44}  
\end{matrix} \right). \label{tw19}
\end{eqnarray}
Where
\begin{eqnarray}
A_{11}&=&-\frac{\partial F_1(\bar{S},\bar{I_1})}{\partial S} -\lambda <0 \nonumber \\
A_{13}&=&-\frac{\partial F_1(\bar{S},\bar{I_1})}{\partial I_1}<0 \nonumber\\
A_{14}&=&-\frac{\partial F_2(\bar{S},0)}{\partial I_2} \nonumber\\
A_{24}&=&-k \bar{V_1}<0\nonumber\\
A_{31}&=&\frac{\partial F_1(\bar{S},\bar{I_1})}{\partial S}>0 \nonumber\\
A_{33}&=&\frac{\partial F_1(\bar{S},\bar{I_1})}{\partial I_1}-\alpha_1= \bar{I_1}\frac{\partial f_1(\bar{S},\bar{I_1})}{\partial {I_1}}+f_1(\bar{S},\bar{I_1})-\alpha_1= \bar{I_1}\frac{\partial f_1(\bar{S},\bar{I_1})}{\partial {I_1}}\leq 0 \nonumber\\
A_{44}&=&\frac{\partial F_2(\bar{S},0)}{\partial I_2}+k \bar{V_1}-\alpha_2=\alpha(\bar{\mathcal{R}_2}-1). \nonumber
\end{eqnarray}
The last equality regarding $ A_{33} $ is that equation \eqref{tw6} implies that $f_1(\bar{S},\bar{I_1})-\alpha_1=0$.
The corresponding characteristic polynomial is
$$p(x)=-(A_{44}-x)(x^3+a_2 x^2+a_1 x+ a_0).$$
Then has an eigenvalue is $A_{44}$ and the remaining ones
satisfy
$$(x^3+a_2 x^2+a_1 x+ a_0)=0.$$
Where
\begin{eqnarray}
a_2&=&-(A_{11}-\mu+A_{33})>0 \nonumber \\
a_1&=&-\mu A_{11}-\mu A_{33}+ A_{11} A_{33}-A_{13} A_{31} \nonumber\\
a_0&=&\mu A_{11} A_{33}-\mu A_{13}A_{31}. \nonumber
\end{eqnarray}
Note that
\begin{eqnarray}
A_{11}A_{33}-A_{13}A_{31}&=&\left(-\frac{\partial F_1(\bar{S},\bar{I_1})}{\partial S} -\lambda \right)\left(\frac{\partial F_1(\bar{S},\bar{I_1})}{\partial I_1}-\alpha_1\right)+\frac{\partial F_1(\bar{S},\bar{I_1})}{\partial I_1}\frac{\partial F_1(\bar{S},\bar{I_1})}{\partial S}\nonumber\\
&=&-\lambda \left(\frac{\partial F_1(\bar{S},\bar{I_1})}{\partial I_1}-\alpha_1\right)+\alpha_1\frac{\partial F_1(\bar{S},\bar{I_1})}{\partial S}>0. \nonumber
\end{eqnarray}
Then $a_1$,$a_0>0$ and
\begin{eqnarray}
a_2a_1-a_0&=&-\left(A_{11}+A_{33}\right)a_1+\mu \left(-\mu A_{11}-\mu A_{33}\right)+\mu \left( A_{11} A_{33}-A_{13} A_{31}\right)-a_0\nonumber\\
&=&-\left(A_{11}+A_{33}\right)a_1+\mu \left(-\mu A_{11}-\mu A_{33}\right)>0.\nonumber
\end{eqnarray}
Applying the Routh–Hurwitz criterion, we see that all roots of $x^3+a_2 x^2+a_1 x+ a_0$ have
negative real parts. If $\bar{\mathcal{R}_2} > 1$, then $A_{44}>0$ therefore $E_1$ is unstable and if $\bar{\mathcal{R}_2} < 1$, then $A_{44}<0$ therefore $E_1$ is stable.
\end{proof}
\begin{rem}\label{remtw3}
$\bar{S}\leq S^0$ and $\bar{V_1}\leq V_1^0$, then $\bar{\mathcal{R}_2}\leq \mathcal{R}_2$, therefore if $\mathcal{R}_2< 1$ then $\bar{\mathcal{R}_2}<1$.
\end{rem}
\begin{teo}\label{teole2}
Let $\tilde{\mathcal{R}_1}=\frac{1}{\alpha_1}\frac{\partial F_1(\tilde{S},0)}{\partial I_1}$. If $\frac{\partial F_2(\tilde{S},\tilde{I_2})}{\partial I_2}\leq 0$ the equilibrium $E_2$ is unstable if $\bar{\mathcal{R}_1} > 1$ while it is locally asymptotically stable if $ 1<\bar{\mathcal{R}_1}$.
\end{teo}
\begin{proof}
Then Eq. \eqref{tw16} at the equilibrium $E_1$  is
\begin{eqnarray}
J_{E_2}&=& \left( \begin{matrix}
B_{11}   & 0 &B_{13} &B_{14} \\
r &B_{22} &0& B_{24} \\
0& 0 & B_{33} & 0\\
B_{41} & B_{42} & 0&B_{44}  \label{tw20}
\end{matrix} \right).
\end{eqnarray}
Where
\begin{eqnarray}
B_{11}&=&-\frac{\partial F_2(\tilde{S},\tilde{I_2})}{\partial S} -\lambda <0 \nonumber \\
B_{13}&=&-\frac{\partial F_1(\tilde{S},0)}{\partial I_1} \nonumber \\
B_{14}&=&-\frac{\partial F_2(\tilde{S},\tilde{I_2})}{\partial I_2} \nonumber\\
B_{22}&=&-\mu-k \tilde{I_2}<0\nonumber\\
B_{24}&=&-k \tilde{V_1}<0\nonumber\\
B_{33}&=&\frac{\partial F_1(\tilde{S},0)}{\partial I_1}-\alpha_1=\alpha_1 \left(\tilde{\mathcal{R}_1}-1\right) \nonumber\\
B_{41}&=&\frac{\partial F_2(\tilde{S},\tilde{I_2})}{\partial S}>0. \nonumber
\end{eqnarray}

\begin{eqnarray}
B_{42}&=&k \tilde{I_2}>0 \nonumber\\
B_{44}&=&\frac{\partial F_2(\bar{S},\bar{I_2})}{\partial I_2}+k \tilde{V_1}-\alpha_2=\tilde{I_2}\frac{\partial f_2(\tilde{S},\tilde{I_2})}{\partial I_2}<0. \nonumber
\end{eqnarray}
The last equality regarding $ B_{44} $ is that equation \eqref{tw10} implies that $k \tilde{V_1}-\alpha_2=-f_2(\tilde{S},\tilde{I_2})$. The corresponding characteristic polynomial is
$$p(x)=-(B_{33}-x)(x^3+b_2 x^2+b_1 x+ b_0)$$
Then \eqref{tw20} has an eigenvalue equal to $B_{33}$ and the remaining ones
satisfy
$$(x^3+b_2 x^2+b_1 x+ b_0)=0.$$
Where
\begin{eqnarray}
b_2&=&-(B_{11}+B_{22}+B_{44})>0. \nonumber \\
b_1&=&B_{22}B_{11}+B_{22}B_{44}+B_{11}B_{44}-B_{14}B_{41}-B_{24}B_{42} \nonumber\\
b_0&=&-B_{22}B_{11}B_{44}-r B_{14} B_{42}+B_{14}B_{22}B_{41}+B_{11}B_{24}B_{42}. \nonumber
\end{eqnarray}
Note that
\begin{eqnarray}
B_{11}B_{44}-B_{14}B_{41}&=&-\lambda \left(\frac{\partial F_2(\bar{S},\bar{I_2})}{\partial I_2}+k \tilde{V_1}-\alpha_2 \right)+\left(-\frac{\partial F_2(\tilde{S},\tilde{I_2})}{\partial S}\right)\left(k \tilde{V_1}-\alpha_2 \right)>0\nonumber.
\end{eqnarray}
And
\begin{eqnarray}
-B_{22}B_{11}B_{44}-r B_{14} B_{42}+B_{14}B_{22}B_{41}&=&\left(\frac{\partial F_2(\tilde{S},\tilde{I_2})}{\partial S} +\lambda \right)\left(k \tilde{V_1}-\alpha_2\right)\left(-\mu-k \tilde{I_2} \right)\nonumber\\
&&-\left(-\mu \right)\left(\frac{\partial F_2(\bar{S},\bar{I_2})}{\partial I_2}\right)\left(-\mu-k \tilde{I_2} \right)\nonumber\\
&&-\left(-r \right)\left(\frac{\partial F_2(\bar{S},\bar{I_2})}{\partial I_2}\right)\left(-\mu\right)>0.\nonumber
\end{eqnarray}
Then $b_1$, $b_0>0$. Also
\begin{eqnarray}
b_2b_1-b_0&=&-B_{44}b_1-B_{22}\left(B_{22}B_{11}+B_{22}B_{44}-B_{24}B_{42}\right)-B_{22}\left( B_{11}B_{44}-B_{14}B_{41}\right)\nonumber\\
&&-B_{11}\left(B_{22}B_{11}+B_{22}B_{44}+B_{11}B_{44}-B_{14}B_{41}\right)+B_{11}B_{24}B_{42}\nonumber\\
&&+B_{22}B_{11}B_{44}+r B_{14} B_{42}-B_{14}B_{22}B_{41}-B_{11}B_{24}B_{42}\nonumber\\
&=&-B_{44}b_1-B_{22}\left(B_{22}B_{11}+B_{22}B_{44}-B_{24}B_{42}\right)\nonumber\\
&&-B_{11}\left(B_{22}B_{11}+B_{22}B_{44}+B_{11}B_{44}-B_{14}B_{41}\right)+r B_{14} B_{42}\geq 0. \nonumber
\end{eqnarray}
Applying the Routh–Hurwitz criterion, we see that all roots of $x^3+b_2 x^2+b_1 x+ b_0$ have
negative real parts. If $\bar{\mathcal{R}_1} > 1$, then $B_{33}>0$ therefore $E_2$ is unstable and if $\bar{\mathcal{R}_1} < 1$, then $B_{33}<0$ therefore $E_2$ is stable.
\end{proof}
\begin{rem}\label{remtw4}
$\tilde{S}\leq S^0$, then $\tilde{\mathcal{R}_1}\leq \mathcal{R}_1$, therefore if $\mathcal{R}_1< 1$ then $\tilde{\mathcal{R}_1}<1$.
\end{rem}

\begin{rem}\label{remtw5}
The theorem \ref{teole2} is valid for $\frac{\partial F_2(\tilde{S},\tilde{I_2})}{\partial I_2}> 0$ if $b_2b_1-b_0>0$ (Note that $b_i>0$ $i=0,1,2$).
\end{rem}
\begin{teo}\label{teoper}
If $\bar{\mathcal{R}_2}>1$ and $\tilde{\mathcal{R}_1}>1$ then system (1) is uniformly persistent. 
\end{teo}
\begin{proof}
The result follows from an application of Theorem 4.6 in [12], with $ X_1 = \text{int} (\mathbb{R}^4_+) $ and $X_2=\text{bd}(\mathbb{R}^4 _+)$ this choice is in accordance by virtue of Lemma \ref{lemaco} there exists a compact set $\Omega$ in which all solution of system \eqref{tw3} initiated in $ \mathbb{R}^4_+ $  ultimately enter and remain forever after. The compactness condition $ C_ {4.2}$ is easily verified for this set
$\Omega_1$. Denoting the omega limit set of the solution $x(t,x_0)$ of system \eqref{tw3} starting in $x_0 \in \mathbb{R}^4_+$ by $w(x_0)$. Note that $w(x_0)$ is bounded (Lemma \ref{lemaco}), we need to determine the following set:     
\begin{eqnarray}
\Omega_2=\bigcup_{y \in Y_2} w(y), \ \text{where} \ Y_2=\left\lbrace x_0 \in X_2| x(t,x_0)\in X_2, \forall t>0 \right\rbrace . \nonumber
\end{eqnarray}
From the system equations \eqref{tw3} it follows that all solutions starting in $\text{bd}(\mathbb{R}^4_+)$ but not on the $I_1$ axis or $I_2$ axis  leave $\text{bd}(\mathbb{R}^4 _+)$ and that the axes $I_1$ and 
$I_2$ are invariant sets, which implies that 
$$Y_2=\left\lbrace (S,V_1,I_1,I_2) \in \text{bd}(\mathbb{R}^4 _+)| I_1=0 \ \text{or} \ I_2=0 \right\rbrace.$$ 
Therefore $\Omega_2=\{E_0,E_1,E_2\}$, then $\bigcup_{i=1}^3 \{E_i\}$ is a covering of $\Omega_2$, which is isolated (since $E_i$ $(i=1,2,3)$ is a saddle point) and acyclic. It will be enough to show that ${E_i}$ (i=1,2,3) is a weak repeller for $X_ 1$. 

By definition $\{E_i\}$ is a weak repeller for $X_1$ if for every solution $(S(t),V_1(t),I_1(t),I_2(t))$ starting in $(S_0,V_{10},I_{10},I_{20}) \in X_1$
 \begin{eqnarray}
\limsup_{t \to +\infty} \|(S(t),V_1(t),I_1(t),I_2(t))-E_i \|>0. \nonumber
 \end{eqnarray}
We will first show that $\{E_0\}$ is a weak repeller for $X_1$, Suppose the claim is false, i.e, there exists a solution $(S(t),V_1(t),I_1(t),I_2(t))$ starting in $(S_0,V_{10},I_{10},I_{20}) \in X_1$ such that
 \begin{eqnarray}
\limsup_{t \to +\infty} \|(S(t),V_1(t),I_1(t),I_2(t))-E_1 \|=0. \nonumber
 \end{eqnarray}
Then exists $T_1>0$ such that for every $\eta_1>0$  
\begin{center}
$S^0-\eta_1<S(t)$, $V_1^0-\eta_1<V_1(t)$, $0<I_1(t)<\eta_1$ and $0<I_2(t)<\eta_1$ $\forall t\geq T_1$ 
\end{center}  
Since $\bar{\mathcal{R}_2}>1$ and $\tilde{\mathcal{R}_1}>1$, then $\mathcal{R}_2=\frac{1}{\alpha_2}\left( f_2(S^0,0)+kV^0\right)>1$ and $\mathcal{R}_1=\frac{1}{\alpha_1}\left( f_1(S^0,0)\right)>1$, therefore $f_2(S^0,0)+kV^0-\alpha_2>0$ and $f_1(S^0,0)-\alpha_1>0$. Because of the continuity of $f_2(S,I_2)+kV_1-\alpha_2$ and $f_1(S,I_1)-\alpha_1$ exist a sufficiently small constant $\eta_2>0$, such that $f_1(S^0-\eta_2 ,\eta_2)-\alpha_1>0$ and $f_2(S^0-\eta_2 ,\eta_2)+k (V_1^0-\eta_2)-\alpha_1>0$.

Let  $\eta_1=\eta_2$, then for $t\geq T_1$.
\begin{eqnarray}
\dot{I_1}&=&I_1\left(f_1(S,I_1)-\alpha_1 \right)\nonumber\\
&\geq& I_1\left(f_1(S^0-\eta_2 ,\eta_2)-\alpha_1 \right).\nonumber
\end{eqnarray} 
and
\begin{eqnarray}
\dot{I_2}&=&I_2\left(f_2(S,I_2)+k V_1-\alpha_2 \right)\nonumber\\
&\geq& I_2\left(f_2(S^0-\eta_2 ,\eta_2)+k (V_1^0-\eta_2)-\alpha_2 \right).\nonumber
\end{eqnarray} 
By comparison principle, we have 
\begin{center}
$I_1(t)\geq I_1(T_1)e^{(f_1(S^0-\eta_2 ,\eta_2)-\alpha_1)(t-T_1)}$ and $I_2(t)\geq I_2(T_1)e^{(f_2(S^0-\eta_2 ,\eta_2)+k (V_1^0-\eta_2)-\alpha_2)(t-T_1)} $, $\forall t\geq T_1$.     
\end{center}
Note that $I_1(T_1)$, $I_2(T_1)>0$, which implies that $\displaystyle \lim_{t \to \infty} I_1=\lim_{t \to \infty} I_2=\infty$, this gives a contradiction. Then $\{E_0\}$ is a weak repeller for $X_1$.

Similarly it is shown  that $ \{E_1 \} $ and $ \{E_2 \}$ are weak repeller for $X_1$. Then we conclude that system \eqref{tw3} is uniformly persistent.    
\end{proof}
Further, it is proved in [13] uniform persistence implies the existence of an interior equilibrium point. Therefore, we have established the following.
\begin{teo}\label{teoee3}
The model \eqref{tw3} admits a endemic equilibrium $E_3=(S^*,V_1^*,I_1^*,I_2^*)$ if $\bar{\mathcal{R}_2}>1$ and $\tilde{\mathcal{R}_1}>1$.
\end{teo}
\begin{teo}\label{teole3}
If $c_1c_2-c_3>0$ and $c_1c_2c_3-c_3^2-c_1^2c_4>0$, where
\begin{eqnarray}
c_1&=&-C_{44}-C_{33}-C_{22}-C_{11}\nonumber \\
c_2&=&-C_{41}C_{14}-C_{42}C_{24}+C_{44}C_{33}+C_{44}C_{22}+C_{44}C_{11}-C_{31}C_{13}+C_{33}C_{22}\nonumber\\
&&+C_{33}C_{11}+C_{22}C_{11} \nonumber
\end{eqnarray}
\begin{eqnarray}
c_3&=&-rC_{42}C_{14}+C_{41}C_{14}C_{33}+C_{41}C_{14}C_{22}+C_{42}C_{24}C_{33}+C_{42}C_{24}C_{11}+C_{44}C_{31}C_{13}\nonumber\\
&&-C_{44}C_{33}C_{22}-C_{44}C_{33}C_{11}-C_{44}C_{22}C_{11}+C_{31}C_{13}C_{22}-C_{33}C_{22}C_{11}\nonumber\\
c_4&=&rC_{42}C_{14}C_{33}-C_{41}C_{14}C_{33}C_{22}+C_{42}C_{24}C_{31}C_{13}-C_{42}C_{24}C_{33}C_{11}\nonumber\\
&&-C_{44}C_{31}C_{13}C_{22}+C_{44}C_{33}C_{22}C_{11}.\nonumber
\end{eqnarray}Then $E_3$ is locally asymptotically stable.
\end{teo}
\begin{proof}
Then Eq. \eqref{tw16} at the equilibrium $E_3$  is
\begin{eqnarray}
J_{E_3}&=& \left( \begin{matrix}
C_{11}   & 0 &C_{13} &C_{14} \\
r & C_{22} &0& C_{24} \\
C_{31}& 0 & C_{33} & 0\\
C_{41} & C_{42} & 0&C_{44}\nonumber  
\end{matrix} \right).
\end{eqnarray}
Where
\begin{eqnarray}
C_{11}&=&-\frac{\partial F_1(S^*,I_1^*)}{\partial S}-\frac{\partial F_2(S^*,I_2^*)}{\partial S} -\lambda <0 \nonumber \\
C_{13}&=&-\frac{\partial F_1(S^*,I_1^*)}{\partial I_1} \nonumber \\
C_{14}&=&-\frac{\partial F_2(S^*,I_2^*)}{\partial I_2} \nonumber\\
C_{22}&=&-\mu -kI_2^*<0 \nonumber\\
C_{24}&=&-k V_1^*<0\nonumber\\
C_{31}&=&\frac{\partial F_1(S,I_1^*)}{\partial S}>0 \nonumber\\
C_{33}&=&\frac{\partial F_1(S^*,I_1^*)}{\partial I_1}-\alpha_1= I_1^*\frac{\partial f_1(S^*,I_1^*)}{\partial {I_1}}+f_1(S^*,I_1^*)-\alpha_1= I_1^*\frac{\partial f_1(S^*,I_1^*)}{\partial {I_1}}\leq 0 \nonumber\\
C_{41}&=&\frac{\partial F_2(S^*,I_2^*)}{\partial S}>0. \nonumber\\
C_{42}&=&k I_2^*>0. \nonumber\\
C_{44}&=&\frac{\partial F_2(S^*,I_2^*)}{\partial I_2}+k V_1^*-\alpha_2=I_2^*\frac{\partial f_2(S^*,I_2^*)}{\partial I_2}\leq 0. \nonumber
\end{eqnarray}
The corresponding characteristic polynomial is
$$p(x)=x^4+c_1x^3+c_2x^2+c_3x+c_4.$$
Note that $c_1>0$,
\begin{eqnarray}
-C_{41}C_{14}+C_{44}C_{11}&=&C_{44}(C_{11}+C_{41})-C_{41}\left(kV_1^*-\alpha_2\right)>0 \nonumber\\
C_{33}C_{11}-C_{31}C_{13}&=&C_{33}(C_{11}+C_{31})-C_{31}\left(-\alpha_1\right)>0.\nonumber
\end{eqnarray}
then $c_2>0$, If $C_{14}\geq 0$ then $c_3>0$ and $c_4>0$, while if $C_{14}<0$ we have that   
\begin{eqnarray}
C_{41}C_{14}C_{33}+C_{44}C_{31}C_{13}-C_{44}C_{33}C_{11}&=&-C_{44}C_{33}(C_{11}+C_{41}+C_{31})-C_{44}(\alpha_1)(C_{31})\nonumber\\
&&-(kV_1^*-\alpha_2)C_{33}(-C_{41})>0\nonumber\\
-rC_{42}C_{14}-C_{44}C_{22}C_{11}+C_{41}C_{14}C_{22}&=&-C_{44}C_{22}(C_{11}+C_{41}+C_{31}+r)-(C_{14})(\mu)(r)\nonumber\\
&&+(kV_1^*-\alpha_2)C_{22}(C_{41}+r)+C_{44}C_{22}C_{31}>0.\nonumber
\end{eqnarray}
and
\begin{eqnarray}
rC_{42}C_{14}C_{33}-C_{44}C_{33}C_{22}C_{11}-C_{41}C_{14}C_{33}C_{22}-C_{44}C_{31}C_{13}C_{22} >0.\nonumber
\end{eqnarray}
Then $c_3>0$ and $c_4>0$. If $c_1c_2-c_3>0$ and $c_1c_2c_3-c_3^2-c_1^2c_4>0$ by Routh–Hurwitz criterion, we see that all roots of $x^4+c_1x^3+c_2x^2+c_3x+c_4$ have negative real parts, then $E_3$ is locally asymptotically stable.
\end{proof}
\subsubsection{Global stability of equilibria}
In this section, we study the global properties of the equilibria. We use Lyapunov function to show the global stabilities. Such Lyapunov functions all take advantage of the properties of the function.
$$g(x)=x-1-ln(x).$$
which is positive in $\mathbb{R}_+$ except at $x = 1$, where it vanishes.
\begin{teo}\label{teoge0}
The DFE $E_0$ is globally asymptotically stable if,
\begin{equation}
\mathcal{R}_0<1.\nonumber
\end{equation}
\end{teo}
\begin{proof}
Consider the Lyapunov function
\begin{equation}
V(S,V_1,I_1,I_2)=I_1+I_2,\nonumber
\end{equation}
Since $I_1 , I_2> 0$, then $V ( S, V_1 , I_1 , I_2) \geq 0$ and $V ( S, V_1 , I_1 , I_2)$ attains zero at $I_1=I_2=0$.

Now, we need to show $\dot{V}<0 $.
\begin{eqnarray}
\dot{V}&=&\dot{I_1}+\dot{I_2}\nonumber \\ 
&=&F_1 (S,I_1)-\alpha_1 I_1+F_2 (S,I_2)+kI_2 V_1-\alpha_2 I_2.\nonumber\\
&=&I_1(f_1 (S,I_1)-\alpha_1)+I_2(f_2 (S,I_2)+k V_1-\alpha_2). \nonumber
\end{eqnarray}
For $S\leq S^0$ and $V_1\leq V_1^0$
\begin{eqnarray}
\dot{V}&\leq & I_1(f_1 (S^0,0)-\alpha_1)+I_2(f_2 (S^0,0)+k V_1^0-\alpha_2).\nonumber\\
&=&I_1\left(\frac{\partial{F_1 (S^0,0)}}{\partial I_1}-\alpha_1\right)+I_2\left(\frac{\partial{F_2 (S^0,0)}}{\partial I_2}+k V_1^0-\alpha_2\right)\nonumber\\
&=&\alpha_1 I_1\left(\mathcal{R}_1-1\right)+\alpha_2 I_2\left(\mathcal{R}_2-1\right)\leq 0.\nonumber
\end{eqnarray}
Furthermore, $\frac{\text{d}v}{\text{d}t}=0$ if and only if $I_1=I_2=0$, so the largest invariant set contained in $\left\lbrace(S,V_1,I_1,I_2)\in \Omega_1 | \frac{dV}{dt}=0  \right\rbrace$ is the hyperplane $I_1=I_2=0$, By LaSalle's invariant principle, this implies that all solution in $\Omega_1$ approach the hyperplane $I_1=I_2=0$ as $t\to \infty$. Also, All solution of \eqref{tw3} contained in such plane satisfy $\dot{S}=\Lambda-\lambda S$, $\dot{V_1}=rS-\mu V_1$, which implies that $S\to \frac{\Lambda}{\lambda}$ and $V_1 \to \frac{r \Lambda}{\mu  \lambda}$ as $t \to \infty$, that is, all of these solution approach $E_0$. Therefore we conclude that $E_0$ is globally asymptotically stable in $\Omega_1$.

Now we will show that every solution $(S(t),V_1(t),I_1(t),I_2(t))\in \mathcal{R}^4_+$, where $t\to \infty$ $(S(t),V_1(t),I_1(t),I_2(t))\in \Omega_1$, let $(S(t),V_1(t),I_1(t),I_2(t))\in \mathcal{R}^4_+$.
Then 
\begin{eqnarray}
\dot{S}&\leq& \Lambda-\lambda S\nonumber
\end{eqnarray}
By the comparison principle $ \displaystyle \lim_{t \to \infty} \sup S(t)\leq \frac{\Lambda}{\lambda}
=S^0$. Then $S(t)\leq S^0$ for $t$ sufficiently large.
 
Also if $S(t)\leq S^0$. 
\begin{eqnarray}
\dot{V_1}&\leq& rS^0-(\mu + k I_2)V_1\leq  rS^0-\mu V_1 \nonumber
\end{eqnarray}
By the comparison principle $ \displaystyle \lim_{t \to \infty} \sup V(t)\leq \frac{r S^0}{\mu}=V_1^0$. Therefore $E_0$ is globally asymptotically stable.
\end{proof}

From now on, we assume that
\begin{itemize}
\item[H4)] For $i=1,2.$ $f_i(S,I_i)=Sg_i(S,I_i)$.
\end{itemize}
\begin{lem}\label{lempGe}
Let $a>0$ be a constant, for $i=1,2$ if $\frac{\partial F_i(S,I_i)}{\partial I_i}\geq 0$, then
\begin{eqnarray}
\left(\frac{I_i}{a}-\frac{F_i(S,I_i)}{F_i(S,a)} \right)\left(\frac{F_i(S,a)}{F_i(S,I_i)}-1 \right)\leq 0 \nonumber
\end{eqnarray}
\end{lem}
\begin{proof}
Note that
\begin{eqnarray}
\left(\frac{I_i}{a}-\frac{F_i(S,I_i)}{F_i(S,a)} \right)\left(\frac{F_i(S,a)}{F_i(S,I_i)}-1 \right)&=&\frac{I_i}{a}\left(1-\frac{f_i(S,I_i)}{f_i(S,a)} \right)\left(\frac{F_i(S,a)}{F_i(S,I_i)}-1 \right)  \nonumber
\end{eqnarray}
If $a\geq I_i$, then
\begin{eqnarray}
\frac{f_i(S,I_i)}{f_i(S,a)}\geq 1 \ \text{and} \ \frac{F_i(S,a)}{F_i(S,I_i)}\geq 1. \nonumber
\end{eqnarray}
If $a\leq I_i$, then
\begin{eqnarray}
\frac{f_i(S,I_i)}{f_i(S,a)}\leq 1 \ \text{and} \ \frac{F_i(S,a)}{F_i(S,I_i)}\leq 1. \nonumber
\end{eqnarray}
Therefore
\begin{eqnarray}
\left(\frac{I_i}{a}-\frac{F_i(S,I_i)}{F_i(S,a)} \right)\left(\frac{F_i(S,a)}{F_i(S,I_i)}-1 \right)\leq 0 \nonumber.
\end{eqnarray}
\end{proof}

\begin{teo}\label{teoge1}
 $E_1$ is globally asymptotically stable if,
\begin{center}
$\mathcal{R}_2<1$.
\end{center}
\end{teo}
\begin{proof}
Consider the Lyapunov function
\begin{equation}
V(S,V_1,I_1,I_2)=I_2,\nonumber
\end{equation}
Since $I_2> 0$, then $V ( S, V_1 , I_1 , I_2) \geq 0$ and $V ( S, V_1 , I_1 , I_2)$ attains zero at $I_2=0$.
Now, we need to show $\dot{V}<0 $.
\begin{eqnarray}
\dot{V}&=&\dot{I_2}\nonumber \\ 
&=&F_2 (S,I_2)+kI_2 V_1-\alpha_2 I_2.\nonumber\\
&=&I_2(f_2 (S,I_2)+k V_1-\alpha_2). \nonumber
\end{eqnarray}
For $S\leq S^0$ and $V_1\leq V_1^0$
\begin{eqnarray}
\dot{V}&\leq & I_2(f_2 (S^0,0)+k V_1^0-\alpha_2).\nonumber\\
&=&I_2\left(\frac{\partial{F_2 (S^0,0)}}{\partial I_2}+k V_1^0-\alpha_2\right)\nonumber\\
&=&\alpha_2 I_2\left(\mathcal{R}_2-1\right)\leq 0.\nonumber
\end{eqnarray}
Furthermore, $\frac{\text{d}V}{\text{d}t}=0$ if and only if $I_2=0$. Suppose that $(S(t), V_1(t), I_1(t),I_2(t))$ is a solution of \eqref{tw3} contained entirely in the set $M =\{(S(t), V_1(t), I_1(t),I_2(t)) \in \Omega_1|\dot{V}=0 \}$. Then, $\dot{I_2}=0$ and, from the above inequalities, we have
$I_2 = 0$. Thus, the largest positively invariant set contained in $M$ is the plane $I_2 = 0$. By
LaSalle’s invariance principle, this implies that all solutions in approach the plane
$I_2 = 0$ as $t\to \infty$. On the other hand, solutions of (4) contained in such plane satisfy
\begin{eqnarray}
\dot{S}&=&\Lambda- F_1 (S,I_1)-\lambda S \nonumber \\
 \dot{V_1}&=&rS-(\mu)V_1\nonumber\\
 \dot{I_1}&=&F_1 (S,I_1)-\alpha_1 I_1. \nonumber
\end{eqnarray}
Now we will show that $S(t)\to \bar{S}$, $V_1(t)\to \bar{V_1}$ and $I_1(t)\to \bar{I_1}$
Consider the Lyapunov function
\begin{equation}
V(S,V_1,I_1)=\int^S_{\tilde{S}} \left(1-\frac{F_1(\bar{S},\bar{I_1})}{F_1(\chi,\bar{I_1})} \right)\text{d}\chi +\bar{I_1}g\left(\frac{I_1}{\bar{I_1}}\right). \nonumber
\end{equation}
Note that $1-\frac{F_1(\bar{S},\bar{I_1})}{F_1(\chi,\bar{I_1})}=\frac{\bar{I_1}(f_1(S,\bar{I_1})-f_1(\bar{S},\bar{I_1}))}{F_1(\chi,\bar{I_1})}$, by H2) $f_1(S,\bar{I_1})-f_1(\bar{S},\bar{I_1})\geq 0$ if $S\geq \tilde{S}$ and $f_1(S,\bar{I_1})-f_1(\bar{S},\bar{I_1})\leq 0$ if $S\leq \tilde{S}$, then $\int^S_{\tilde{S}} \left(1-\frac{F_1(\bar{S},\bar{I_1})}{F_1(\chi,\bar{I_1})} \right)\text{d}\chi\geq 0$ for all $S$. Therefore, $V ( S, V_1 , I_1) \geq 0$ and $V ( S, V_1 , I_1)$ attains zero at $S(t)=\bar{S},$ and $I_1(t)=\bar{I_1}$. 

Now, we need to show $\dot{V}<0 $.
\begin{eqnarray}
\dot{V}&=&\left(1-\frac{F_1(\bar{S},\bar{I_1})}{F_1(S,\bar{I_1})} \right)\dot{S}+\left(1-\frac{\bar{I_1}}{I_1} \right)\dot{I_1}\nonumber \\ 
&=&\left(1-\frac{F_1(\bar{S},\bar{I_1})}{F_1(S,\bar{I_1})} \right)\left(\Lambda - F_1(S,I_1)-\lambda S \right)+\left(1-\frac{\bar{I_1}}{I_1} \right)(F_1(S,I_1)-\alpha_1 I_1)\nonumber\\
&=&\left(1-\frac{F_1(\bar{S},\bar{I_1})}{F_1(S,\bar{I_1})} \right)\left(\lambda \bar{S}+F_1(\bar{S},\bar{I_1}) - F_1(S,I_1)-\lambda S \right)\nonumber \\
&&+F_1(S,I_1)-\alpha_1 I_1- \bar{I_1}f_1(S,I_1)+\alpha_1 \bar{I_1} \nonumber
\end{eqnarray}
\begin{eqnarray}
&=&\lambda(\bar{S}-S)\left(1-\frac{F_1(\bar{S},\bar{I_1})}{F_1(S,\bar{I_1})} \right)+\left(1-\frac{F_1(\bar{S},\bar{I_1})}{F_1(S,\bar{I_1})} \right)F_1(\bar{S},\bar{I_1})- F_1(S,I_1)\nonumber \\
&&+\frac{F_1(\bar{S},\bar{I_1})}{F_1(S,\bar{I_1})} F_1(S,I_1)+F_1(S,I_1)-\frac{I_1 F_1(\bar{S},\bar{I_1})}{\bar{I_1}} - \bar{I_1}f_1(S,I_1)+F_1(\bar{S},\bar{I_1}) \nonumber\\
&=&\left(2-\frac{F_1(\bar{S},\bar{I_1})}{F_1(S,\bar{I_1})}+\frac{F_1(S,I_1)}{F_1(S,\bar{I_1})} -\frac{I_1}{\bar{I_1}} - \frac{\bar{I_1}f_1(S,I_1)}{F_1(\bar{S},\bar{I_1})}\right)F_1(\bar{S},\bar{I_1}) \nonumber\\
&&+\lambda(\bar{S}-S)\left(1-\frac{F_1(\bar{S},\bar{I_1})}{F_1(S,\bar{I_1})} \right).\nonumber
\end{eqnarray}
Note that
\begin{eqnarray}
\lambda(\bar{S}-S)\left(1-\frac{F_1(\bar{S},\bar{I_1})}{F_1(S,\bar{I_1})} \right)=\lambda(\bar{S}-S)\left(1-\frac{f_1(\bar{S},\bar{I_1})}{f_1(S,\bar{I_1})} \right)\leq 0. \nonumber
\end{eqnarray}
and 
\begin{scriptsize}
\begin{eqnarray}
2-\frac{F_1(\bar{S},\bar{I_1})}{F_1(S,\bar{I_1})}+\frac{F_1(S,I_1)}{F_1(S,\bar{I_1})} -\frac{I_1}{\bar{I_1}} - \frac{\bar{I_1}f_1(S,I_1)}{F_1(\bar{S},\bar{I_1})}&=&2-\frac{F_1(\bar{S},\bar{I_1})}{F_1(S,\bar{I_1})}+\frac{F_1(S,I_1)}{F_1(S,\bar{I_1})} -\frac{I_1}{\bar{I_1}} \nonumber\\
&&- \frac{\bar{I_1}F_1(S,I_1)}{I_1F_1(\bar{S},\bar{I_1})}+1-\frac{F_1(S,I_1)F_1(S,\bar{I_1})}{F_1(S,I_1)F_1(S,\bar{I_1})}\nonumber\\
&&+\frac{IF_1(S,\bar{I_1})}{\bar{I_1}F_1(S,I_1)}-\frac{IF_1(S,\bar{I_1})}{\bar{I_1}F_1(S,I_1)}\nonumber\\
&=&3-\frac{F_1(\bar{S},\bar{I_1})}{F_1(S,\bar{I_1})}-\frac{\bar{I_1}F_1(S,I_1)}{I_1F_1(\bar{S},\bar{I_1})}-\frac{IF_1(S,\bar{I_1})}{\bar{I_1}F_1(S,I_1)}\nonumber\\
&+&\left(\frac{I_1}{\bar{I_1}}-\frac{F_1(S,I_1)}{F_1(S,\bar{I_1})} \right)\left(\frac{F_1(S,\bar{I_1})}{F_1(S,I_1)}-1 \right)\leq 0.\nonumber
\end{eqnarray}
\end{scriptsize}
Then $\dot{V}\leq 0$. Furthermore, $\frac{\text{d}V}{\text{d}t}=0$ if and only if $S=\bar{S}$ and $I_1=\bar{I_1}$, which implies that $S\to \bar{S}$, $I_1 \to \bar{I_1}$ and $I_2\to 0$ as $t \to \infty$. By LaSalle's invariant principle, this implies that all solution in $\Omega_1$ approach the plane $S= \bar{S}$, $I_1= \bar{I_1}$ and $I_2= 0$ as $t\to \infty$. Also, All solution of \eqref{tw3} contained in such plane satisfy $\dot{V_1}=r\bar{S}-\mu V_1$, which implies that $V_1 \to \frac{r\bar{S}}{\mu}=\bar{V_1}$ as $t \to \infty$, that is, all of these solution approach $E_1$. Therefore we conclude that $E_1$ is globally asymptotically stable in $\Omega_1$.
\end{proof}

\begin{teo}\label{teoge2}
$E_2$ is globally asymptotically stable if,
\begin{center}
$\mathcal{R}_1<1$ and $2-\frac{F_2(\tilde{S},\tilde{I_2})}{F_2(S,\tilde{I_2})}+\frac{S F_2(\tilde{S},\tilde{I_2})}{\tilde{S} F_2(S,\tilde{I_2})}-\frac{V_1}{\tilde{V_1}}-\frac{S\tilde{V_1}}{\tilde{S}V_1}\leq 0.$
\end{center}
\end{teo}

\begin{proof}
Consider the Lyapunov function
\begin{equation}
V(S,V_1,I_1,I_2)=I_1. \nonumber
\end{equation}
 Since $I_1> 0$, then $V ( S, V_1 , I_1 , I_2) \geq 0$ and $V ( S, V_1 , I_1 , I_2)$ attains zero at $I_1=0$.
Now, we need to show $\dot{V}<0 $.
\begin{eqnarray}
\dot{V}&=&\dot{I_1}\nonumber \\ 
&=&F_1 (S,I_1)-\alpha_1 I_1\nonumber\\
&=&I_1(f_1 (S,I_1)-\alpha_1) \nonumber
\end{eqnarray}
For $S\leq S^0$ 
\begin{eqnarray}
\dot{V}&\leq & I_1(f_1 (S^0,0)-\alpha_1)\nonumber\\
&=&I_1\left(\frac{\partial{F_1 (S^0,0)}}{\partial I_1}-\alpha_1\right)=\alpha_1 I_1\left(\mathcal{R}_1-1\right)\leq 0.\nonumber
\end{eqnarray}
Furthermore, $\frac{\text{d}v}{\text{d}t}=0$ if and only if $I_1=0$. Suppose that $(S(t), V_1(t), I_1(t),I_2(t))$ is a solution of \eqref{tw3} contained entirely in the set $M =\{(S(t), V_1(t), I_1(t),I_2(t)) \in \Omega_1|\dot{V}=0 \}$. Then, $\dot{I_1}=0$ and, from the above inequalities, we have $I_1 = 0$. Thus, the largest positively invariant set contained in $M$ is the plane $I_1 = 0$. By
LaSalle’s invariance principle, this implies that all solutions in approach the plane
$I_1 = 0$ as $t\to \infty$. On the other hand, solutions of \eqref{tw3} contained in such plane satisfy.
\begin{eqnarray}
\dot{S}&=&\Lambda- F_2 (S,I_2)-\lambda S \nonumber \\
 \dot{V_1}&=&rS-(\mu-k I_2)V_1\nonumber\\
 \dot{I_2}&=&F_2 (S,I_2)-kV_1 I_2-\alpha_1 I_2. \nonumber
\end{eqnarray}
Now we will show that $S(t)\to \tilde{S}$, $V_1(t)\to \tilde{V_1}$ and $I_1(t)\to \tilde{I_1}$
Consider the Lyapunov function
\begin{equation}
V(S,V_1,I_2)=\int^S_{\tilde{S}} \left(1-\frac{F_2(\tilde{S},\tilde{I_2})}{F_2(\chi,\tilde{I_2})} \right)\text{d}\chi+\tilde{V_1}g\left(\frac{V_1}{\tilde{V_1}}\right) +\tilde{I_2}g\left(\frac{I_2}{\tilde{I_2}}\right). \nonumber
\end{equation}
Now, we need to show $\dot{V}<0 $.
\begin{eqnarray}
\dot{V}&=&\left(1-\frac{F_2(\tilde{S},\tilde{I_2})}{F_2(S,\tilde{I_2})} \right)\dot{S}+\left(1-\frac{\tilde{V_1}}{V_1} \right)\dot{V_1}+\left(1-\frac{\tilde{I_2}}{I_2} \right)\dot{I_2}\nonumber \\ 
&=&\left(1-\frac{F_2(\tilde{S},\tilde{I_2})}{F_2(S,\tilde{I_2})} \right)\left(\Lambda - F_2(S,I_2)-\lambda S \right)+\left(1-\frac{\tilde{V_1}}{V_1} \right)(rS-(\mu+kI_2)V_1)\nonumber\\
&&+\left(1-\frac{\tilde{I_2}}{I_2} \right)(F_2(S,I_2)+kI_2V_1-\alpha_2 I_2)\nonumber\\
&=&\left(1-\frac{F_2(\tilde{S},\tilde{I_2})}{F_2(S,\bar{I_2})} \right)\left(\lambda \bar{S}+F_2(\tilde{S},\tilde{I_2}) - F_2(S,I_2)-\lambda S \right)+rS-(\mu+kI_2)V_1\nonumber \\
&&-r\frac{S\tilde{V_1}}{V_1}+(\mu+kI_2)\tilde{V_1}+F_2(S,I_2)+kI_2V_1-\alpha_2 I_2- \tilde{I_2}f_2(S,I_2)-k\tilde{I_2}V_1+\alpha_2 \tilde{I_2} \nonumber\\
&=&\mu(\tilde{S}-S)\left(1-\frac{F_2(\tilde{S},\tilde{I_2})}{F_2(S,\tilde{I_2})} \right)+r\left(\tilde{S}-\tilde{S}\frac{F_2(\tilde{S},\tilde{I_2})}{F_2(S,\tilde{I_2})}-S+S\frac{F_2(\tilde{S},\tilde{I_2})}{F_2(S,\tilde{I_2})} \right)\nonumber \\
&&+\left(1-\frac{F_2(\tilde{S},\tilde{I_2})}{F_2(S,\tilde{I_2})} \right)F_2(\tilde{S},\tilde{I_2})+\frac{F_2(\tilde{S},\tilde{I_2})}{F_2(S,\tilde{I_2})} F_2(S,I_2)+rS \nonumber\\
&&-\frac{r \tilde{S}}{\tilde{V_1}} V_1-r\frac{S\tilde{V_1}}{V_1}+r \tilde{S}-\frac{I_2 F_2(\tilde{S},\tilde{I_2})}{\tilde{I_2}}- \tilde{I_2}f_2(S,I_2)+F_2(\tilde{S},\tilde{I_2}) \nonumber\\
&=&\mu(\tilde{S}-S)\left(1-\frac{F_2(\tilde{S},\tilde{I_2})}{F_2(S,\tilde{I_2})} \right)+r\tilde{S}\left(2-\frac{F_2(\tilde{S},\tilde{I_2})}{F_2(S,\tilde{I_2})}+\frac{S F_2(\tilde{S},\tilde{I_2})}{\tilde{S} F_2(S,\tilde{I_2})}-\frac{V_1}{\tilde{V_1}}-\frac{S\tilde{V_1}}{1\tilde{S}V_1} \right) \nonumber\\
&&+\left(2-\frac{F_2(\tilde{S},\tilde{I_2})}{F_2(S,\tilde{I_2})}+\frac{F_2(S,I_2)}{F_2(S,\tilde{I_2})} -\frac{I_2}{\tilde{I_2}} - \frac{\tilde{I_2}f_2(S,I_2)}{F_2(\tilde{S},\tilde{I_2})}\right)F_2(\tilde{S},\tilde{I_2}). \nonumber
\end{eqnarray}
Note that
\begin{eqnarray}
\mu(\tilde{S}-S)\left(1-\frac{F_2(\tilde{S},\tilde{I_2})}{F_2(S,\tilde{I_2})} \right)\leq 0. \nonumber
\end{eqnarray}
\begin{eqnarray}
2-\frac{F_2(\tilde{S},\tilde{I_2})}{F_2(S,\tilde{I_2})}+\frac{F_2(S,I_2)}{F_2(S,\tilde{I_2})} -\frac{I_2}{\tilde{I_2}} - \frac{\tilde{I_2}f_2(S,I_2)}{F_2(\tilde{S},\tilde{I_2})}\leq 0\nonumber
\end{eqnarray}
We conclude $\dot{V}<0$. Therefore $E_3$ is globally asymptotically stable.
\end{proof}
\begin{rem}\label{remtw6}
Note that if $ \frac{\partial g_2(S,I_2)}{\partial S}\geq 0$, then
\begin{small}
\begin{eqnarray}
2-\frac{F_2(\tilde{S},\tilde{I_2})}{F_2(S,\tilde{I_2})}+\frac{S F_2(\tilde{S},\tilde{I_2})}{\tilde{S} F_2(S,\tilde{I_2})}-\frac{V_1}{\tilde{V_1}}-\frac{S\tilde{V_1}}{\tilde{S}V_1}&=&3-\frac{V_1}{\tilde{V_1}}-\frac{S\tilde{V_1}}{\tilde{S}V_1}-\frac{\tilde{S}}{S}\nonumber\\
&&+\left(-1+\frac{\tilde{S}}{S}\right)\left(1-\frac{g_2(\tilde{S},\tilde{I_2})}{g_2(S,\tilde{I_2})}\right)\leq 0. \nonumber
\end{eqnarray}
\end{small}
\end{rem}
\begin{teo}\label{teoge3}
$E_3$ is globally asymptotically stable if 
\begin{center}
$F_1(S^*,I_1^*)\left(2 - \frac{S^*}{S}- \frac{S g_1(S,I_1)}{S^* g_1(S^*,I_1^*)} \right)+F_2(S^*,I_2^*) \left(2-\frac{S^*}{S}-\frac{S g_2(S,I_2)}{S^* g_2(S^*,I_2^*)}\right)+r S^*\left( 3-\frac{S^*}{S}-\frac{V_1}{V_1^*}-\frac{S V_1^*}{S^* V_1}\right)+\mu S^* \left( 2- \frac{S^*}{S}-\frac{S}{S^*} \right)  +I_1 \left( S^* g_1(S,I_1)-\alpha_1 \right)+I_2 \left( {S}^* g_2(S,I_2)+ k  {V_1}^*-\alpha_2 \right)<0.$ 
\end{center}
\end{teo}
\begin{proof}
Assume $E_3$ exists. Consider the Lyapunov function
\begin{equation}
V(S,V_1,I_1,I_2)=S^* g\left(\frac{S}{S^*} \right)+{V_1}^*g\left(\frac{V_1}{{V_1}^*}\right)+{I_1}^*g\left(\frac{I_1}{{I_1}^*}\right)+{I_2}^*g\left(\frac{I_2}{{I_2}^*}\right). \nonumber
\end{equation}

Where $g(x) = x -1 -ln(x)$. Then $V ( S, V_1 , I_1 , I_2) \geq 0$ and $V ( S, V_1 , I_1 , I_2)$ attains zero at $E_3$.
 
Now, we need to show $\dot{V}<0$.
\begin{eqnarray}
\dot{V}&=&\left(1-\frac{{S}^*}{S} \right)\dot{S}+\left(1-\frac{{V_1}^*}{V_1} \right)\dot{V_{1}}+\left(1-\frac{{I_1}^*}{I_1} \right)\dot{I_1}+\left(1-\frac{{I_2}^*}{I_2} \right)\dot{I_2}\nonumber\\ 
&=&\left(1-\frac{{S}^*}{S} \right)\left(\Lambda- F_1(S,I_1)- F_2(S,I_2)-\lambda S \right)+\left(1-\frac{{V_1}^*}{V_1} \right)\left(rS-(\mu + k I_2)V_1 \right) \nonumber \\
&&+\left(1-\frac{{I_1}^*}{I_1} \right)(F_1(S,I_1)-\alpha_1 I_1)+\left(1-\frac{{I_2}^*}{I_2} \right)\left( F_2(S,I_2)+kI_2 V_1-\alpha_2 I_2\right) \nonumber\\
&=&\Lambda- F_1(S,I_1)- F_1(S,I_1) -\lambda S -\Lambda\frac{{S}^*}{S}+ I_1 {S}^*g_1(S,I_1)+ I_2S^* g_2(S,I_2) +\lambda {S}^*\nonumber \\
&&+rS-\mu V_1- k I_2 V_1 -rS\frac{{V_1}^*}{V_1}+\mu {V_1}^*+ k I_2 {V_1}^*+F_1(S,I_1)-\alpha_1 I_1-{I_1}^* f_1(S,I_1)\nonumber\\
&&+\alpha_1 {I_1}^*+F_2(S,I_2)+kI_2 V_1-\alpha_2 I_2-{I_2}^* f_2(S,I_2)-k{I_2}^* V_1+\alpha_2 {I_2}^* \nonumber
\end{eqnarray}
\begin{eqnarray}
&=&\left(F_1(S^*,I_1^*)+F_2(S^*,I_2^*) +\lambda S^* \right) -\lambda S -\left(F_1(S^*,I_1^*)+F_2(S^*,I_2^*) +\lambda S^* \right) \frac{{S}^*}{S} \nonumber \\
&&+ I_1 S^* g_1(S,I_1)+ I_2{S}^* g_2(S,I_2) +\lambda {S}^*+rS-\mu V_1-rS\frac{{V_1}^*}{V_1}+\mu {V_1}^*+ k I_2 {V_1}^*\nonumber\\
&&-\alpha_1 I_1- I_1^* f_1(S,I_1) +F_1(S^*,I_1^*)-\alpha_2 I_2-I_2^*f_2(S,I_2)-k{I_2}^* V_1+F_2(S^*,I_2^*) \nonumber\\
&&+kI_2^*V_1^*\nonumber\\ 
&=&\left(2 F_1(S^*,I_1^*)-F_1(S^*,I_1^*) \frac{S^*}{S}- I_1^* f_1(S,I_1)\right)+ \left(2 F_2(S^*,I_2^*)-F_2(S^*,I_2^*)  \frac{S^*}{S}\right)   \nonumber \\
&&-I_2^*f_2(S,I_2)+\left( 2\lambda S^*-\lambda S^* \frac{S^*}{S}-\lambda S+rS-rS\frac{{V_1}^*}{V_1}+rS^*-rS^*\frac{V_1}{V_1^*} \right)  \nonumber\\
&&+\left( I_1 S^* g_1(S,I_1)-\alpha_1 I_1\right)+\left( I_2{S}^* g_2(S,I_2)+ k I_2 {V_1}^*-\alpha_2 I_2\right) \nonumber\\
&=&F_1(S^*,I_1^*)\left(2 - \frac{S^*}{S}- \frac{S g_1(S,I_1)}{S^* g_1(S^*,I_1^*)} \right)+F_2(S^*,I_2^*) \left(2-\frac{S^*}{S}-\frac{S g_2(S,I_2)}{S^* g_2(S^*,I_2^*)}\right)   \nonumber \\
&&+r S^*\left(3-\frac{S^*}{S}-\frac{V_1}{V_1^*}-\frac{S V_1^*}{S^* V_1}\right)+\mu S^* \left( 2- \frac{S^*}{S}-\frac{S}{S^*} \right)  \nonumber\\
&&+I_1 \left(  S^* g_1(S,I_1)-\alpha_1 \right)+I_2 \left( {S}^* g_2(S,I_2)+ k  {V_1}^*-\alpha_2 \right). \nonumber
\end{eqnarray}
By the relation of geometric and arithmetic means, we conclude $\dot{V}<0$. Therefore $E_3$ is globally asymptotically stable.
\end{proof}

\section{Numerical simulations}
\begin{figure}
\centering
\includegraphics[scale=.45]{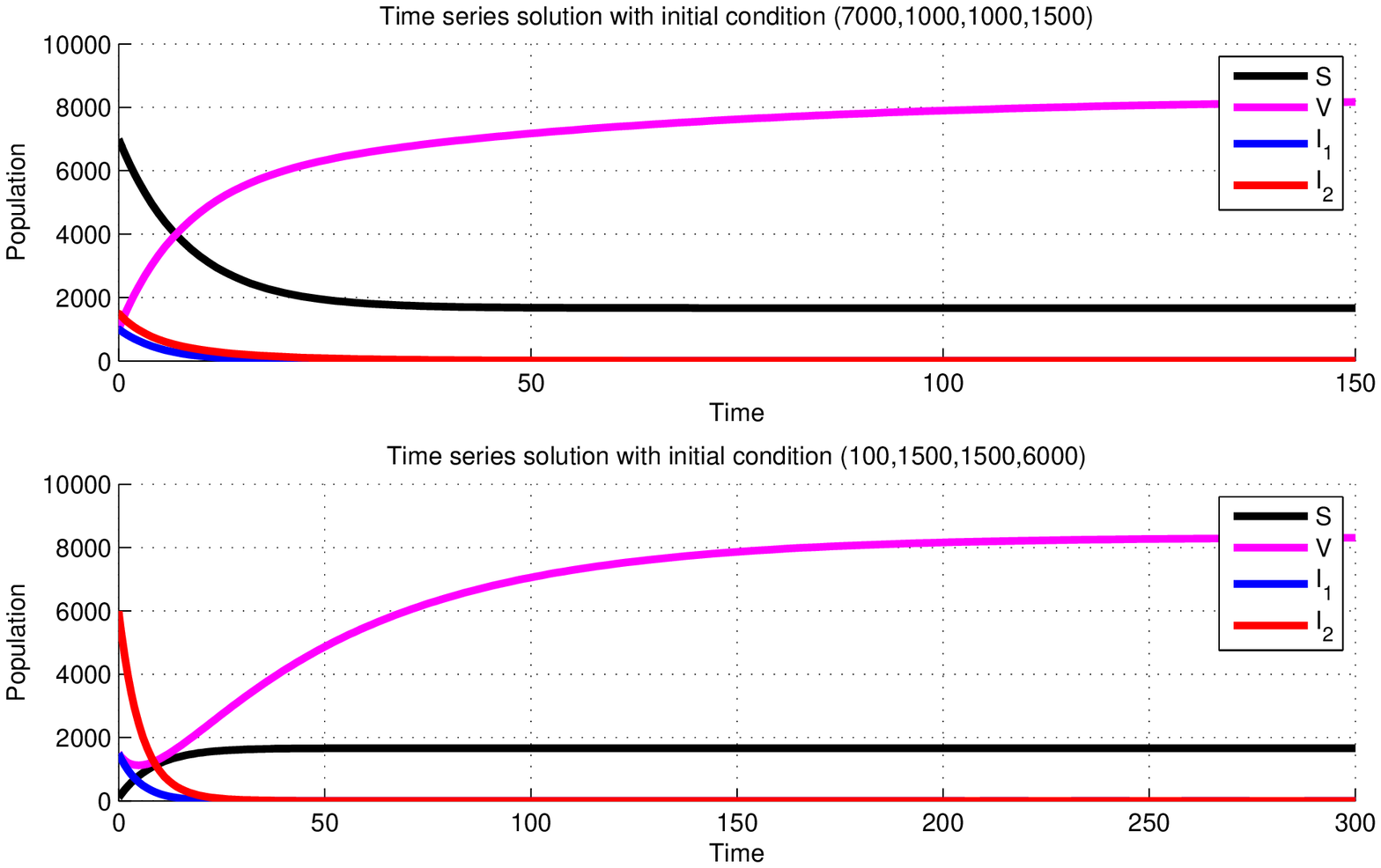}
\caption{Numerical simulation of \eqref{tw3} indicates that $E_0$ is globally asymptotically stable.}
\end{figure}
\begin{figure}
\centering
\includegraphics[scale=.45]{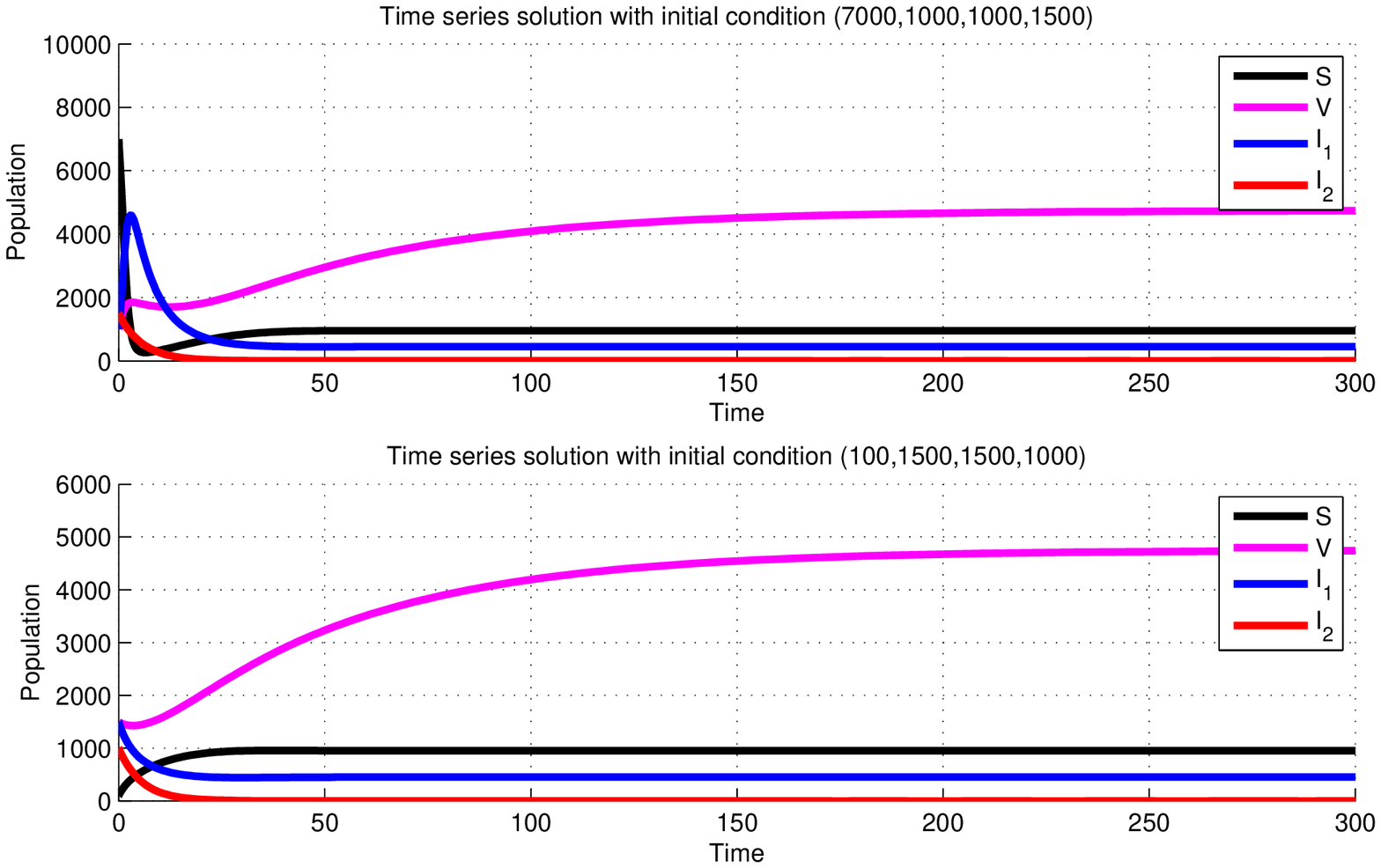}
\caption{Numerical simulation of \eqref{tw3} indicates that $E_1$ is globally asymptotically stable.}
\end{figure}

\begin{figure}
\centering
\includegraphics[scale=.45]{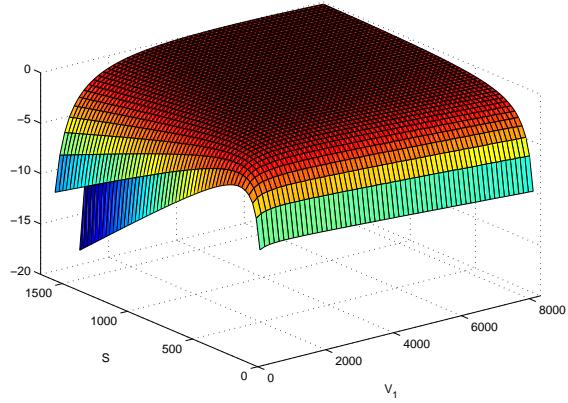}
\caption{Graph of $2-\frac{F_2(\tilde{S},\tilde{I_2})}{F_2(S,\tilde{I_2})}+\frac{S F_2(\tilde{S},\tilde{I_2})}{\tilde{S} F_2(S,\tilde{I_2})}-\frac{V_1}{\tilde{V_1}}-\frac{S\tilde{V_1}}{\tilde{S}V_1}$.}
\end{figure}

\begin{figure}
\centering
\includegraphics[scale=.45]{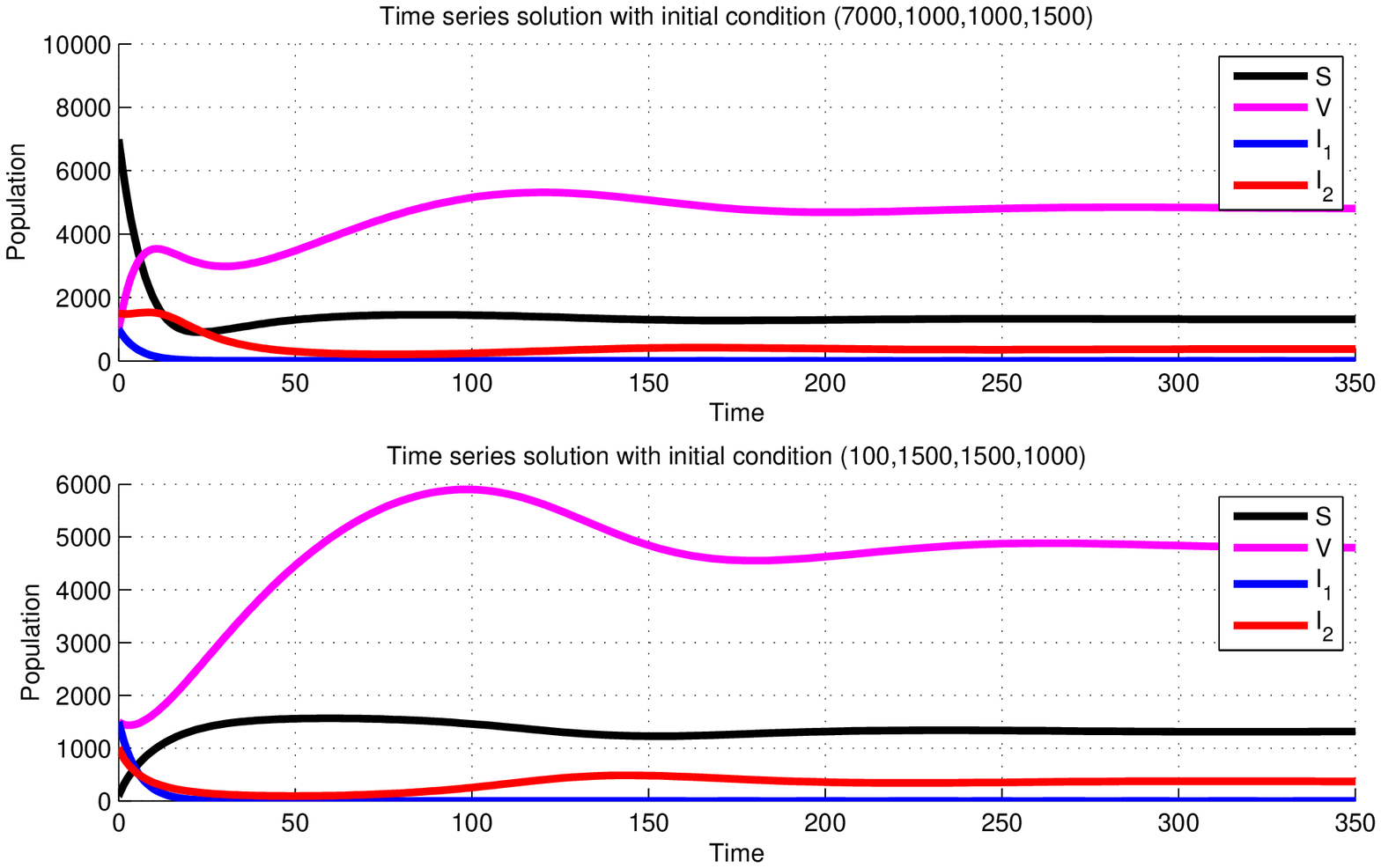}
\caption{Numerical simulation of \eqref{tw3} indicates that $E_2$ is globally asymptotically stable.}
\end{figure}

\begin{figure}
\centering
\includegraphics[scale=.45]{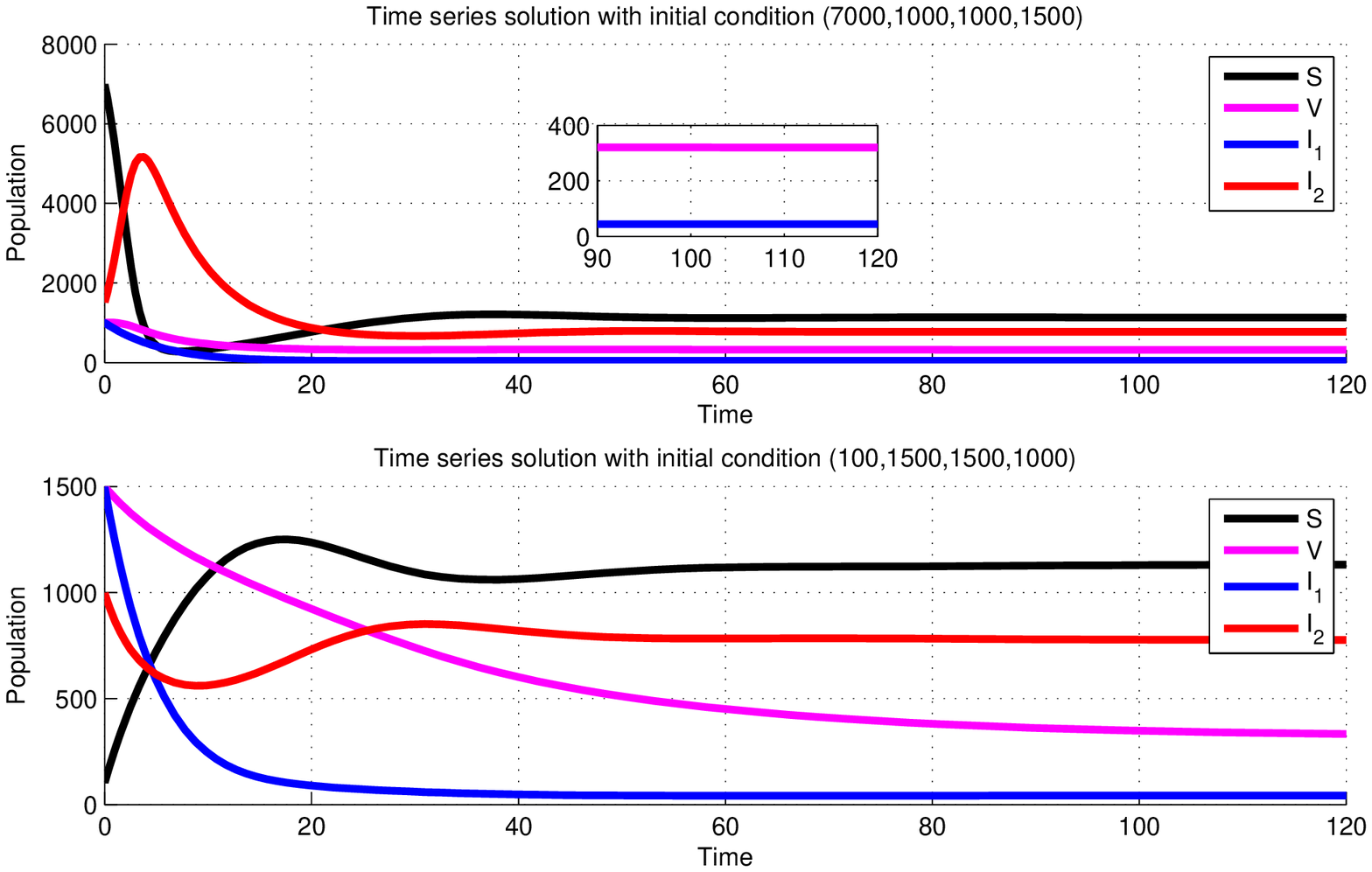}
\caption{Numerical simulation of \eqref{tw3} indicates that $E_3$ is globally asymptotically stable.}
\end{figure}

In this section, we present some numerical simulations of the solutions for system
\eqref{tw3} to verify the results obtained in section 4 and give examples to illustrate theorems in section 5. In system \eqref{tw3}, we set:

\begin{center}
$\displaystyle F_1(S,I_1)=\frac{\beta_1 SI_1}{1+\zeta_1 I_1^2}$, $\displaystyle F_2(S,I_2)=\frac{\beta_2 SI_2}{1+\zeta_2 S}$, $\Lambda=200$, $\gamma_1=0.07$, $\gamma_2=0.09$, $\mu=0.02$, $v_1=0.1$, $v_2=0.1$ and $k=0.00002$.
\end{center}
In this case 
\begin{center}
$\displaystyle g_1(S,I_1)=\frac{\beta_1}{1+\zeta_1 I_1^2}$, $\displaystyle g_2(S,I_1)=\frac{\beta_2}{1+\zeta_2 S}$, $\mathcal{R}_1=\frac{\beta_1 \Lambda}{\alpha_1\lambda}$ and $\mathcal{R}_2=\frac{\beta_2 \Lambda}{\alpha_2(\lambda+\zeta \Lambda)}+\frac{k r \Lambda}{\alpha_2 \mu \lambda}$. 
\end{center}
\begin{itemize}
\item[•] \textbf{Example 6.1.} In system \eqref{tw3}, we set $\beta_1=0.00003$, $r=0.1$, $\beta_2=0.0002$, $\zeta_1=0.7$ and $\zeta_2=0.9$. Then $S^0\approx 1667$, $V^0\approx 8333$, $\mathcal{R}_1\approx0.2632$  $\mathcal{R}_2\approx0.7947$. By theorem \ref{teoge0}, we see that the disease-free equilibrium $E_0$ is globally asymptotically stable. Numerical simulation illustrates our result (see Fig. 1).
\item[•]\textbf{Example 6.2.} In system \eqref{tw3}, we set $\beta_1=0.0002$, $r=0.1$, $\beta_2=0.0002$, $\zeta_1=0$  and $\zeta_2=0.9$. Then $\bar{S}\approx 950$, $\bar{V_1}\approx 4737$, $\bar{I_1}\approx 253$, $\mathcal{R}_1\approx1.7544$,  $\mathcal{R}_2\approx0.7947$. By theorem \ref{teoge1}, we see that the $E_1$ is globally asymptotically stable. Numerical simulation illustrates our result (see Fig. 2). 
\item[•]\textbf{Example 6.3.} In system \eqref{tw3}, we set $\beta_1=0.00003$, $r=0.1$, $\beta_2=0.0002$, $\zeta_1=0.7$  and $\zeta_2=0.001$. Then $\tilde{S}\approx 1314$, $\tilde{V_1}\approx 4814$, $\tilde{I_2}\approx 368$, $\mathcal{R}_1\approx 0.2632$,  $\mathcal{R}_2\approx 1.3889$ and $2-\frac{F_2(\tilde{S},\tilde{I_2})}{F_2(S,\tilde{I_2})}+\frac{S F_2(\tilde{S},\tilde{I_2})}{\tilde{S} F_2(S,\tilde{I_2})}-\frac{V_1}{\tilde{V_1}}-\frac{S\tilde{V_1}}{\tilde{S}V_1}\leq 0$ (see Fig.3). By theorem \ref{teoge2}, we see that the $E_2$ is globally asymptotically stable. Numerical simulation illustrates our result (see Fig. 4).
\item[•]\textbf{Example 6.4.} In system \eqref{tw3}, we set $\beta_1=0.0002$, $r=0.01$, $\beta_2=0.0002$,  $\zeta_1=0.0001$ and $\zeta_2=0.0001$. Then $\mathcal{R}_1\approx 7.0175$,  $\mathcal{R}_2\approx 4.1270$, $\tilde{S}\approx 1134$, $\bar{S}\approx 5310$, $\bar{V_1} \approx 2655$, $\bar{R_2}\approx 3.555$ and $\tilde{R_1}\approx 1.194$. Then by theorem \ref{teoee3}, $E_3=(S^*,V_1^*,I_1^*,I_2^*)$ exists ($S^*\approx 1133$, $V_1^*\approx 320$, $I_1^* \approx 44$, $I_2^* \approx 774$), Also $c_1\approx 0.2501$ $c_2\approx  0.0171$ $c_3\approx 3.4759\times 10^{-04}$ $c_4\approx 3.4759\times 3.9242 10^{-06}$, $c_1 c_2-c_3^2\approx  0.0043$ and $c_1c_2c_3-c_3^2-c_1^2c_4\approx  1.1218e\times 10^{-06}$ by theorem \ref{teole3}, $E_3$ is locally asymptotically stable. Also $ E_3 $ satisfies {\small $F_1(S^*,I_1^*)\left(2 - \frac{S^*}{S}- \frac{S g_1(S,I_1)}{S^* g_1(S^*,I_1^*)} \right)+F_2(S^*,I_2^*) \left(2-\frac{S^*}{S}-\frac{S g_2(S,I_2)}{S^* g_2(S^*,I_2^*)}\right)+\mu S^* \left( 2- \frac{S^*}{S}-\frac{S}{S^*} \right)  +I_1 \left( S^* g_1(S,I_1)-\alpha_1 \right)+I_2 \left( {S}^* g_2(S,I_2)+ k  {V_1}^*-\alpha_2 \right)<0$}. By theorem \ref{teoge3}, we see that the $E_3$ is globally asymptotically stable. Numerical simulation illustrates our result (see Fig. 5).  
\end{itemize}
\section{Concluding remarks}
In this paper,we studied a system of ordinary differential equations to model the disease dynamics
of two strains of influenza with only one vaccination for strain 1 being implemented, and general incidence rate for strain 1 and strain 2. We obtained four equilibrium points:
\begin{itemize}
\item $E_0$ disease free equilibrium, $I_1$   and $I_2 $ are both zero.
\item $E_1$ single-strain-infection equilibria, $I_2 $ are zero.
\item $E_2$ single-strain-infection-equilibria, $I_1 $ are zero.
\item $E_3$ double-strain-infection equilibrium, $I_1$   and $I_2 $ are both positive.
\end{itemize}

We have investigated the topics of existence and non-existence of various equilibria and their stabilities. We also used next generation matrix method to obtain two threshold quantities $\mathcal{R}_1$ and $\mathcal{R}_2$, called the basic reproduction ratios for strain 1 and 2 respectively. It was shown that the global stability of each of the equilibrium points depends on the magnitude of these threshold quantities.  More precisely, we have proved the following:
\begin{itemize}
\item If $\mathcal{R}_0<1$ the disease free equilibrium $E_0$ is globally asymptotically stable. If $\mathcal{R}_0>1$, then $E_0$ is unstable.  
\item If $\mathcal{R}_1>1$ the model \eqref{tw3} admits a single-strain-infection-equilibria $E_1$. Also if $\mathcal{R}_2<1$ then $E_1$ is globally asymptotically stable.
\item If $\mathcal{R}_2>1$ the model \eqref{tw3} admits a single-strain-infection equilibria $E_2$. Also if 
$2-\frac{F_2(\tilde{S},\tilde{I_2})}{F_2(S,\tilde{I_2})}+\frac{S F_2(\tilde{S},\tilde{I_2})}{\tilde{S} F_2(S,\tilde{I_2})}-\frac{V_1}{\tilde{V_1}}-\frac{S\tilde{V_1}}{\tilde{S}V_1}<0$,  then $E_2$ is globally asymptotically stable.
\item If $\bar{\mathcal{R}_2}>1$ and $\tilde{\mathcal{R}_1}>1$ the model \eqref{tw3}  admits a double strain infection equi\-li\-brium $E_3$. Also if 
{\small $F_1(S^*,I_1^*)\left(2 - \frac{S^*}{S}- \frac{S g_1(S,I_1)}{S^* g_1(S^*,I_1^*)} \right)+F_2(S^*,I_2^*) \left(2-\frac{S^*}{S}-\frac{S g_2(S,I_2)}{S^* g_2(S^*,I_2^*)}\right)$ $+r S^*\left( 3-\frac{S^*}{S}-\frac{V_1}{V_1^*}-\frac{S V_1^*}{S^* V_1}\right)+\mu S^* \left( 2- \frac{S^*}{S}-\frac{S}{S^*} \right)  +I_1 \left( S^* g_1(S,I_1)-\alpha_1 \right)+I_2 \left( {S}^* \right.$ $\left.g_2(S,I_2)+k  {V_1}^*-\alpha_2 \right)<0$}. Then $E_3$ is globally asymptotically stable.
\end{itemize} 
In order to discuss the meaning of our mathematical results, let us rewrite the two key indirect parameters $\mathcal{R}_1$ and $\mathcal{R}_2$ in terms of the direct model parameters as shown below:
\begin{equation}
\mathcal{R}_1=\frac{f_1\left(\frac{\Lambda}{r+\mu},0 \right)}{\alpha_1}, \ \ \ \mathcal{R}_2=\frac{f_2\left(\frac{\Lambda}{r+\mu},0 \right)}{\alpha_2}+\frac{k r \Lambda}{\alpha_2 \mu (r+\mu)} \nonumber 
\end{equation}
Also the derivative of $\mathcal{R}_2$ with respect to $r$ is,
\begin{equation}
\frac{\Lambda}{\alpha_2(r+\mu)^2}\left(-\frac{\partial f_2\left(\frac{\Lambda}{r+\mu},0\right)}{\partial S}+k \right) \nonumber
\end{equation}
Note that $\mathcal{R}_1(r)$ is decreasing and $\mathcal{R}_2(r)$ depends on  $\frac{\partial f_2\left(\frac{\Lambda}{\mu},0\right)}{\partial S}$. Now we will analyse some cases of incidence rate.
\begin{itemize}
\item[(C1)]$F_i(S,I)=\beta_i SI_i$, then $\frac{\partial f_2\left(\frac{\Lambda}{\mu},0\right)}{\partial S}=\beta_i$. 
\item[(C2)]$F_i(S,I)=\frac{\beta_i SI_i} {1 + \zeta_i S}$, then $\frac{\partial f_2\left(\frac{\Lambda}{\mu},0\right)}{\partial S}=\frac{\beta_i} {1 + \zeta_i \left(\frac{\Lambda}{r+\mu} \right)}$.
\item[(C3)]$F_i(S,I)=\frac{\beta_i SI_i} {1 + \zeta_i I_i^2}$, then $\frac{\partial f_2\left(\frac{\Lambda}{\mu},0\right)}{\partial S}=\beta_i$.
\end{itemize}
Note that for (C1) and (C3) $\mathcal{R}_2(r)$ is increasing if $\beta_i<k$,  $\mathcal{R}_2(r)$ is decreasing if $\beta_i>k$ and $\mathcal{R}_2(r)$ is constant if $\beta_i=k$. For (C2) $\mathcal{R}_2(r)$ is increasing if $\beta_i\leq k$ ($\zeta\neq 0$). If $\beta_i> k$ $\mathcal{R}_2(r)$ is increasing if $\frac{\zeta_i k \Lambda}{\beta_i-k}-\mu<r$ and decreasing if $\frac{\zeta_i k \Lambda}{\beta_i-k}-\mu>r$.

Also for if the force of infection of strain 1 is (C2), then
$\mathcal{R}_1=\frac{\beta_1}{\alpha_1(1 + \zeta_1 S^0)}$, note that $\mathcal{R}_1$ is decreasing in $\zeta_1$. If the force of infection of strain 2 is (C2), then $\mathcal{R}_2=\frac{\beta_2  \Lambda}{\alpha_2(\lambda+\zeta_2 \Lambda)}+\frac{k r \Lambda}{\alpha_2 \mu \lambda}$, note that $\mathcal{R}_2$ is decreasing in $\zeta_2$.

With the above information and the results in Section 5, the vaccination is always beneficial for controlling strain 1, its impact on strain 2 depends on the force of infection. If the forced of infection of strain 2 is (C2), the impact of vaccination depends of values of $\beta_2$, $k$ and $\zeta_2$.  If $\zeta_2=0$; if $\beta_2 > k$ it plays a positive role, and if $\beta_2 < k$, it has a negative impact in controlling strain 2. This is reasonable because larger $k$ (than $\beta_2$) means that vaccinated individuals are more likely to be infected by strain 2 than those who are not vaccinated, and thus, is helpful to strain 2. Smaller $k$ (than $\beta_2$) implies the opposite. If $\zeta_2 \neq 0$; if $\beta_2 > k$, it plays a positive role and if $\beta_2 < k$, not necessarily has a negative impact in controlling strain 2, because $\mathcal{R}_2(\zeta_2)$ is decreasing, i.e., for $ \zeta_2 $ sufficiently large it can play a positive role. This is reasonable because larger $k$ (than $\beta_2$) means that vaccinated individuals are more likely to be infected by strain 2 than those who are not vaccinated, but if $\zeta_2$  is large it means that the population is taking precautions to avoid the infection of strain 2.

Finally, we remark that our model can be improved and generalized. For example, the model can be modified to contain two vaccinations, also we can consider the effect of time delay on vaccine-induced immunity and incorporate the diffusion of individuals. We leave these problems for future investigation.

\textbf{Acknowledgments}This work was supported by Sistema Nacional de Investigadores (15284) and Conacyt-Becas.
\section*{References} 
\begin{itemize}
\item[1)]  Influenza (Seasonal). (2018). Retrieved 21 April 2019, from https://www.who.int/news-room/fact-sheets/detail/influenza-(seasonal).

\item[2)] Rahman A. and Zou X. Flu epidemics: a two strain flu model with a single vaccination. Journal of Biological Dynamics 5,376-390 (2011).

\item[3)]Medina M., Vintiñi E., Villena J., Raya R. and Alvarez S. Lactococcus lactisas an adjuvant and delivery vehicle of antigens against pneumococcal respiratory infections. Bioengineered Bugs 1, 313-325 (2010).

\item[4)]Chowell G., Ammon C., Hengartner N. and Hyman, J. Transmission dynamics of the great influenza pandemic of 1918 in Geneva, Switzerland: Assessing the effects of hypothetical interventions. Journal Of Theoretical Biology 241, 193-204 (2006). 

\item[5)]Mills C., Robins J. and Lipsitch M. Transmissibility of 1918 pandemic influenza. Nature 432, 904-906 (2004).

\item[6)] Cauchemez S., Valleron A., Boëlle P., Flahault A. and Ferguson N. Estimating the impact of school closure on influenza transmission from Sentinel data. Nature 452, 750-754 (2008). 

\item[7)] Capasso V. and Serio G. A Generalization of the Kermack-Mckendrick deterministic epidemic model,  Mathematical Biosciences 42, 43-61 (1978).

\item[8)]Baba I. and Hincal E. A model for influenza with vaccination and awareness. Chaos, Solitons and Fractals  106, 49–55 (2018).

\item[9)]Baba I. and Hincal E. Global stability analysis of two-strain epidemic model with
bilinear and non-monotone incidence rates. The European Physical Journal Plus 132 (2017).

\item[10)]Wang L., Zhang X. and Liu Z. An SEIR Epidemic Model with Relapse and General Nonlinear Incidence Rate with Application to Media Impact. Qualitative Theory Of Dynamical Systems 17, 309-329 (2017). 
\item[11)]Van den Driessche P. and Watmough J. Reproduction numbers and sub-threshold
endemic equilibria for compartmental models of disease transmission, Mathematical Biosciences 180, 29-48 (2002).
\item[12)]Thieme H.R. , Persistence under relaxed point-dissipativity (with application to an endemic model), SIAM Journal on Mathematical Analysis 24, 407-435 (1993).
\item[13)]Butler G. J., Freedman H. I. and Waltman P. Uniformly persistent systems. Proceedings of the American Mathematical Society 96, 425-430 (1986).
\end{itemize}
\end{document}